\documentclass{amsart}

\usepackage{amssymb}
\usepackage{amsfonts}
\usepackage{amsmath}
\usepackage{amsthm}
\usepackage{graphicx}
\usepackage{mathrsfs}
\usepackage{dsfont}
\usepackage{amscd}
\usepackage{multirow}
\usepackage[all]{xy}
\usepackage[scaled]{helvet}
\usepackage{courier}
\normalfont
\usepackage[T1]{fontenc}
\usepackage{calligra}
\usepackage{verbatim}
\usepackage[usenames,dvipsnames]{color}
\usepackage[colorlinks=true,linkcolor=Blue,citecolor=Violet]{hyperref}
\usepackage{enumerate}

\newcommand{\N}{{\mathds{N}}}
\newcommand{\Z}{{\mathds{Z}}}

\newcommand{\R}{{\mathds{R}}}
\newcommand{\C}{{\mathds{C}}}
\newcommand{\T}{{\mathds{T}}}

\newcommand{\D}{{\mathfrak{D}}}
\newcommand{\A}{{\mathfrak{A}}}
\newcommand{\B}{{\mathfrak{B}}}

\newcommand{\Lip}{{\mathsf{L}}}

\newcommand{\modpropinquity}[1]{{\mathsf{\Lambda}^{\mathsf{mod}}_{#1}}}

\newcommand{\Kantorovich}[1]{{\mathsf{mk}_{#1}}}
\newcommand{\KantorovichMod}[1]{{\mathsf{k}_{#1}}}

\newcommand{\Haus}[1]{{\mathsf{Haus}_{#1}}}

\newcommand{\StateSpace}{{\mathscr{S}}}

\newcommand{\MongeKant}{{Mon\-ge-Kan\-to\-ro\-vich metric}}

\newcommand{\Lqcms}{{\JLL} quantum compact metric space}

\newcommand{\gQqcms}{quasi-Leibniz quantum compact metric space}

\newcommand{\gQVB}{metrized quantum vector bundle}

\newcommand{\qcms}{quantum compact metric space}

\newcommand{\unit}{1}

\newcommand{\sa}[1]{{\mathfrak{sa}\left({#1}\right)}}

\newcommand{\inner}[3]{{\left<{#1},{#2}\right>_{#3}}}
\newcommand{\JLL}{Lei\-bniz}

\newcommand{\bridge}[1]{#1}
\newcommand{\dom}[1]{{\operatorname*{dom}\left({#1}\right)}}

\newcommand{\norm}[2]{{\left\|{#1}\right\|_{#2}}}

\newcommand{\bridgemodularreach}[1]{{\varrho^\sharp\left({#1}\right)}}

\newcommand{\bridgeimprint}[1]{{\varpi\left(\bridge{#1}\right)}}
\newcommand{\bridgelength}[1]{{\lambda\left(\bridge{#1}\right)}}

\newcommand{\bridgenorm}[2]{{\mathsf{bn}_{ \bridge{#1}  }\left({#2}\right)}}

\newcommand{\CDN}{{\mathsf{D}}}
\newcommand{\cocycle}[1]{{\mathrm{e}_{#1}}}
\newcommand{\worknote}[1]{} 
\newcommand{\opnorm}[3]{{\left|\mkern-1.5mu\left|\mkern-1.5mu\left| {#1} \right|\mkern-1.5mu\right|\mkern-1.5mu\right|_{#3}^{#2}}}

\newcommand{\alg}[1]{{\mathfrak{#1}}}

\newcommand{\modlip}[2]{{\module{D}_{#1}\left({#2}\right)}}
\newcommand{\module}[1]{{\mathscr{#1}}}
\newcommand{\qt}[1]{{\mathcal{A}_{#1}}}
\newcommand{\HeisenbergGroup}{{\mathds{H}_3}}
\newcommand{\HeisenbergMod}[2]{{\module{H}_{#1}^{#2}}}

\theoremstyle{plain}
\newtheorem{theorem}{Theorem}[section]

\newtheorem{lemma}[theorem]{Lemma}

\newtheorem{theorem-definition}[theorem]{Theorem-Definition}

\theoremstyle{definition}
\newtheorem{definition}[theorem]{Definition}

\newtheorem{convention}[theorem]{Convention}

\theoremstyle{remark}

\newtheorem{notation}[theorem]{Notation}

\renewcommand{\geq}{\geqslant}
\renewcommand{\leq}{\leqslant}

\numberwithin{equation}{section}

\allowdisplaybreaks[4]

\begin{document}

\title[Heisenberg Modules and the modular propinquity]{Convergence of Heisenberg Modules over Quantum $2$-tori for the Modular Gromov-Hausdorff Propinquity}
\author{Fr\'{e}d\'{e}ric Latr\'{e}moli\`{e}re}
\email{frederic@math.du.edu}
\urladdr{http://www.math.du.edu/\symbol{126}frederic}
\address{Department of Mathematics \\ University of Denver \\ Denver CO 80208}
\thanks{This work is part of the project supported by the grant H2020-MSCA-RISE-2015-691246-QUANTUM DYNAMICS and grant \#3542/H2020/2016/2 of the Polish Ministry of Science and Higher Education.}

\date{\today}

\begin{abstract}
The modular Gromov-Hausdorff propinquity is a distance on classes of modules endowed with quantum metric information, in the form of a metric form of a connection and a left Hilbert module structure. This paper proves that the family of Heisenberg modules over quantum two tori, when endowed with their canonical connections, form a continuous family for the modular propinquity.
\end{abstract}
\maketitle



\section{Introduction}

The modular Gromov-Hausdorff propinquity \cite{Latremoliere16c} is a distance on modules over C*-algebras endowed with some quantum metric information, designed to advance our project of constructing an analytic framework for the study of classes of C*-algebras as geometric objects. Convergence for compact quantum metric has been an active area of our research with many developments, built on top of our noncommutative analogue of the Gromov-Hausdorff metric called the dual propinquity \cite{Latremoliere13,Latremoliere13b,Latremoliere14,Latremoliere15,Latremoliere15b,Latremoliere15c,Latremoliere15d,Latremoliere16b}. A natural question in noncommutative metric geometry concerned the behavior of modules over convergent sequences of quantum metric spaces. Indeed, modules encode much geometric information and are key ingredients for theories in mathematical physics, and thus for our research program to advance toward its goal, it is essential to develop a metric geometric theory for modules over C*-algebras. Remarkably, there seems to be no classical model from which to work for such a new distance. We propose our answer, called the modular Gromov-Hausdorff propinquity, as a deep extension of the quantum propinquity in \cite{Latremoliere16c}, and this paper provides the first deep example of convergence for our new metric. We proved in \cite{Latremoliere15c} that quantum tori form a continuous family for the quantum propinquity, and we now prove that Heisenberg modules over quantum two-tori also form continuous families for our new modular propinquity. As all finitely generated, projective modules over quantum two tori are sums of Heisenberg modules and free modules, and free modules were handled in \cite{Latremoliere16c}, we thus provide in this paper a detailed picture of the metric geometry of the class of finitely generated, projective modules over quantum two-tori.

Noncommutative metric geometry \cite{Connes89,Rieffel98a,Rieffel00} studies noncommutative generalizations of Lipschitz algebras, defined as:
\begin{definition}
An ordered pair $(\A,\Lip)$ is a \emph{\Lqcms} when $\A$ is a unital C*-algebra, whose norm we denote as $\|\cdot\|_\A$ and whose unit is denoted as $\unit_\A$, and $\Lip$ is a seminorm defined on a dense Jordan-Lie subalgebra $\dom{\Lip}$ of the space of self-adjoint elements $\sa{\A}$ of $\A$ such that:
\begin{enumerate}
\item $\{a\in\dom{\Lip} : \Lip(a) = 0 \} = \R\unit_\A$,
\item the {\MongeKant} $\Kantorovich{\Lip}$ defined on the state space $\StateSpace(\A)$ of $\A$ by setting, for any two $\varphi,\psi \in \StateSpace(\A)$:
\begin{equation*}
\Kantorovich{\Lip}(\varphi,\psi) = \sup\left\{ |\varphi(a) - \psi(a)| : a\in\dom{\Lip}, \Lip(a) \leq 1 \right\}
\end{equation*}
metrizes the weak* topology restricted to $\StateSpace(\A)$,
\item $\Lip$ is lower semi-continuous,
\item $\max\left\{ \Lip\left(\frac{a b + b a}{2}\right),\Lip\left(\frac{a b - b a}{2i}\right)\right\} \leq \|a\|_\A \Lip(b) + \|b\|_\A \Lip(a)$.
\end{enumerate}
\end{definition}

{\Lqcms s}, and more generally {\gQqcms s} (a generalization we will not need in this paper), form a category with the appropriate notion of Lipschitz morphisms \cite{Latremoliere16}, containing such important examples as quantum tori \cite{Rieffel98a}, Connes-Landi spheres \cite{li03}, group C*-algebras for hyperbolic groups and nilpotent groups \cite{Rieffel02,Ozawa05}, AF algebras \cite{Latremoliere15c}, Podl{\`e}s spheres \cite{Kaad18}, certain C*-crossed-products \cite{Hawkins13}, among others. Any compact metric space $(X,d)$ give rise to the {\Lqcms} $(C(X),\mathrm{Lip})$ where $C(X)$ is the C*-algebra of $\C$-valued continuous functions over $X$, and $\mathrm{Lip}$ is the Lipschitz seminorm induced by $d$.

Our interest in this paper is to prove the continuity of families of modules over {\Lqcms s}. A module over an {\Lqcms} is defined as follows. We refer to \cite{Latremoliere16c} for the motivation behind this definition.

\begin{definition}[{\cite[Definition 3.8]{Latremoliere16c}}]\label{qvb-dev}
A \emph{(Leibniz) \gQVB} $(\module{M},\inner{\cdot}{\cdot}{\module{M}},\CDN,\A,\Lip)$ consists of:
\begin{enumerate}
\item a {\Lqcms} $(\A,\Lip)$ called the base space,
\item a $\A$-left Hilbert module $(\module{M},\inner{\cdot}{\cdot}{\module{M}})$,
\item a norm $\CDN$ defined on a dense subspace of $\module{M}$ such that $\CDN(\omega) \geq \sqrt{\norm{\inner{\omega}{\omega}{\module{M}}}{\A}}$ for all $\omega\in \module{M}$, and such that the set:
\begin{equation*}
\left\{ \omega \in \module{M} : \CDN(\omega) \leq 1 \right\}
\end{equation*}
is compact in $\module{M}$,
\item for all $a \in \sa{\A}$ and for all $\omega \in \module{M}$, we have the \emph{inner Leibniz inequality}:
\begin{equation*}
\CDN_{\module{M}}\left(a\omega\right) \leq \left(\|a\|_\A+\Lip_\A(a)\right) \CDN_{\module{M}}(\omega)\text{,}
\end{equation*}
\item for all $\omega,\eta \in \module{M}$, we have the \emph{modular Leibniz inequality}:
\begin{equation*}
\max\left\{\Lip_\A\left(\Re\inner{\omega}{\eta}{\module{M}}\right),\Lip_\A\left(\Im\inner{\omega}{\eta}{\module{M}}\right)\right\} \leq 2 \CDN_{\module{M}}(\omega) \CDN_{\module{M}}(\eta) \text{.}
\end{equation*}
\end{enumerate}
The norm $\CDN$ is called a \emph{D-norm}.
\end{definition}
We note that we index inner products of Hilbert modules by the module instead of the usual convention to index by the base C*-algebra.

We introduce in \cite{Latremoliere16c} (see \cite{Latremoliere18d} as well) a metric on the class of {\gQVB s}, called the \emph{modular propinquity} $\modpropinquity{}$. We refer to \cite{Latremoliere16c} for the construction of $\modpropinquity{}$, which is indeed a metric up to full quantum isometry of {\gQVB s}, and some of its core properties.

In this paper, we will compute an upper bound on the modular propinquity between certain {\gQVB s} constructed from Heisenberg modules over quantum tori, and we now recall the key elements of our construction in \cite{Latremoliere16c} which we will use to derive these bounds.

First, given any two {\gQVB s}, we use a structure to connect them together, thus enabling us to quantify how far apart they might be. This structure is called a modular bridge, as it extends the idea of a bridge between {\qcms s} introduced in \cite{Latremoliere13}.

\begin{definition}\label{modular-bridge-def}
Let $\Omega_\A = (\module{M},\inner{\cdot}{\cdot}{\module{M}}, \CDN_{\module{M}},\A,\Lip_\A)$ and $\Omega_\B = (\module{N},\inner{\cdot}{\cdot}{\module{N}}, \CDN_{\module{N}},\B,\Lip_\B)$ be two {\gQVB s}. A \emph{modular bridge}
\begin{equation*}
\gamma = (\D,x,\pi_\A,\pi_\B,(\omega_j)_{j\in J}, (\eta_j)_{j\in J})
\end{equation*}
from $\Omega_\A$ to $\Omega_\B$ is given by:
\begin{enumerate}
\item a unital C*-algebra $\D$,
\item an element $x \in \D$, called the \emph{pivot} of $\gamma$, such that:
\begin{equation*}
\StateSpace_1(\D|\omega) := \left\{ \varphi \in \StateSpace(\D) \middle\vert \forall a\in\sa{\D} \quad \varphi(a\omega) = \varphi(\omega a) = \varphi(a) \right\}
\end{equation*}
is not empty,
\item two unital *-monomorphisms $\pi_\A : \A \hookrightarrow \D$ and $\pi_\B : \B \hookrightarrow \D$,
\item an index set $J$ and two families $(\omega_j)_{j\in J} \in \module{M}^J$, $(\eta_j)_{j\in J}) \in \module{N}^J$, respectively called the \emph{anchors} and \emph{co-anchors} of $\gamma$, such that for all $j\in J$, we have $\CDN_{\module{M}}(\omega_j)\leq 1$ and $\CDN_{\module{N}}(\eta_j) \leq 1$.
\end{enumerate}
\end{definition}

In \cite{Latremoliere13,Latremoliere16c}, we associate a number, called the length, to a modular bridge. This number involves the notion of length for a bridge between {\qcms s} as defined in \cite{Latremoliere13}, though we will not need its exact expression in this paper.

\begin{definition}\label{bridge-length-def}
Let $\Omega_\A = (\module{M},\inner{\cdot}{\cdot}{\module{M}}, \CDN_{\module{M}},\A,\Lip_\A)$ and $\Omega_\B = (\module{N},\inner{\cdot}{\cdot}{\module{N}}, \CDN_{\module{N}},\B,\Lip_\B)$ be two {\gQVB s} and let:
\begin{equation*}
\gamma = (\D,x,\pi_\A,\pi_\B,(\omega_j)_{j\in J},(\eta_j)_{j\in J})
\end{equation*}
be a bridge from $\Omega_\A$ to $\Omega_\B$.

The \emph{length} $\bridgelength{\gamma}$ of $\gamma$ is the maximum of the length of $(\D,x,\pi_\A,\pi_\B)$ as defined in \cite[Definition 3.17]{Latremoliere13} and $\bridgeimprint{\gamma} + \bridgemodularreach{\gamma}$ where:
\begin{enumerate}
\item the modular reach $\bridgemodularreach{\gamma}$ of $\gamma$ is:
\begin{equation*}
\bridgemodularreach{\gamma} = \max\left\{ \bridgenorm{\gamma}{\inner{\omega_j}{\omega_k}{\module{M}}, \inner{\eta_j}{\eta_k}{\module{N}}} : j,k \in J \right\}\text{,}
\end{equation*}
with
\begin{equation*}
\forall (a,b)\in\A\oplus\B \quad \bridgenorm{\gamma}{a,b} = \left\| \pi_\A(a) x - x \pi_\B(b) \right\|_\D \text{;}
\end{equation*}

\item the imprint $\bridgeimprint{\gamma}$ of $\gamma$ is:
\begin{multline*}
\bridgeimprint{\gamma} = \max\left\{ \Haus{\module{M}}(\{\omega_j:j\in J\}, \{\omega\in\module{M} : \CDN_{\module{M}}(\omega)\leq 1 \}),\right. \\ \left. \Haus{\module{N}}(\{\eta_j:j\in J\}, \{\eta\in\module{N} : \CDN_{\module{N}}(\eta)\leq 1 \}) \right\}\text{,}
\end{multline*}
where $\Haus{\module{M}}$ is the Hausdorff metric induced by the metric:
\begin{equation*}
\KantorovichMod{\CDN_{\module{M}}} : \omega,\eta \in \module{M} \mapsto \sup \left\{\left|\inner{\omega - \eta}{\xi}{\module{M}}\right| : \xi\in\module{M},\CDN_{\module{M}}(\xi) \leq 1 \right\}\text{,}
\end{equation*} 
with a similar definition for $\Haus{\module{N}}$.
\end{enumerate}
\end{definition}

For our purpose, the key role of bridges is summarized in the fact that by construction, if $\mathds{A}$ and $\mathds{B}$ are two {\gQVB s}, and if $\gamma$ is a bridge from $\mathds{A}$ to $\mathds{B}$, then:
\begin{equation}\label{propinquity-eq}
  \modpropinquity{}\left(\mathds{A},\mathds{B}\right) \leq \bridgelength{\gamma}\text{.}
\end{equation}

The present manuscript provides our first main example of convergence for the modular propinquity: the family of Heisenberg modules over quantum $2$-tori. As shown in \cite{Rieffel83}, finitely generated projective modules over irrational rotation algebras can be described, up to module isomorphism, as either free --- a case with which we dealt in \cite{Latremoliere16c} --- or constructed through a projective unitary representation of $\R^2$, or equivalently, a unitary representation of the Heisenberg group,  as seen in \cite{Connes80}. We briefly indicate in the next section how to endow Heisenberg modules with the structure of a {\gQVB}, which is the topic of \cite{Latremoliere17a} where we discussed the relation between Connes' connection \cite{Connes80,ConnesRieffel87} on Heisenberg modules and our quantum metric structure.

As they are {\gQVB s}, we can ask whether Heisenberg modules form continuous families for the modular propinquity. More specifically, Heisenberg modules over a quantum torus are naturally parametrized by their $K_0$ class, or equivalently, by the value the trace of that class, which is an element of $\Z + \theta\Z \subseteq \R$. We will prove in this paper that if $p,q \in \Z$ are fixed, then the Heisenberg modules whose $K_0$ class has trace $p\theta+q$ vary continuously in $\theta$ for the modular propinquity (note that the base algebra, the quantum torus, also varies continuously with $\theta$ for the quantum propinquity).

The strategy of this paper, which is reflected in its structure, begins with proving a continuous field type of result for the D-norms defined on Heisenberg modules. We then prove a form of uniform approximation for elements in Heisenberg modules using certain convolution-type operators. This involve the use of some harmonic analysis based upon the Hermite functions. We then bring all of this together in our main result.

As a matter of convention throughout this paper, we will use the following notations.

\begin{notation}
By default, the norm of a normed vector space $E$ is denoted by $\norm{\cdot}{E}$. When $\A$ is a C*-algebra, the space of self-adjoint elements of $\A$ is denoted by $\sa{\A}$. The state space of $\A$ is denoted by $\StateSpace(\A)$. In this work, all C*-algebras $\A$ will always be unital with unit $\unit_\A$.
\end{notation}

\begin{convention}
If $P$ is some seminorm on a vector subspace $D$ of a vector space $E$, then for all $x\in E\setminus D$ we set $P(x) = \infty$. With this in mind, the domain $D$ of $P$ is the set $\{ x \in E : P(x) < \infty \}$, with the usual convention that $0\infty = 0$ while all other operations involving $\infty$ give $\infty$.
\end{convention}

\section{Background on Heisenberg modules as {\gQVB s}}

A \emph{($2$-dimensional) quantum torus} $\qt{\theta}$, for $\theta\in\R$, is the twisted convolution C*-algebra $C^\ast\left(\Z^2,\cocycle{\theta}\right)$ of $\Z^2$ for (restriction to $\Z^2$ of) the $2$-cocycle (of $\R^2$):
\begin{equation}\label{sigma-eq}
  \cocycle{\theta} : ((x_1,y_1),(x_2,y_2)) \in \R^2\times \R^2 \longmapsto \exp\left(i\pi\theta(x_2 y_1 - x_1 y_2)\right)\text{.}
\end{equation}
It is thus the universal C*-algebra generated by two unitaries $U_\theta$, $V_\theta$ such that $V_\theta U_\theta = \cocycle{\theta}((1,0),(0,1)) U_\theta V_\theta$. By universality, there exists a (strongly continuous) action $\beta_\theta$ of $\T^2 = \{(\exp(ix),\exp(iy)):x,y\in\R\}$ on $\qt{\theta}$ by *-automorphisms, uniquely defined by setting for all $(z,w)\in\T^2$:
\begin{equation*}
  \beta_\theta^{z,w}U_\theta=z U_\theta\text{ and }\beta_\theta^{z,w}V_\theta = w V_\theta\text{.}
\end{equation*}

Quantum tori are primary examples of noncommutative Riemannian manifolds \cite{Connes80,Rieffel90}. The question which we solve in this manuscript is whether interesting modules over quantum tori form continuous families for the modular propinquity. Arguably, the most relevant modules to consider are the \emph{Heisenberg modules} over the quantum tori, the construction of which we now summarize.

The Heisenberg group $\HeisenbergGroup$ is defined as $\R^3$ with the multiplication:
\begin{equation*}
  \forall (x,y,t),(x',y',t') \in \HeisenbergGroup \quad (x,y,t)(x',y',t') = (x + x', y + y', t + t' + x y') \text{.}
\end{equation*}

We follow the construction in \cite{ConnesRieffel87}. Let us fix $p \in \Z$, $q\in \N\setminus \{0\}$ and $d \in q\N\setminus\{0\}$. Set $\eth = \theta - \frac{p}{q}$.

The irreducible unitary representation $\alpha_{\eth,d}$ of $\HeisenbergGroup$ on $L^2(\R)\otimes\C^d = L^2(\R,\C^d)$ is given by:
\begin{equation}\label{alpha-action-eq}
\alpha_{\eth,d}^{x,y,u} \xi : s\in \R \mapsto \exp(2i\pi(\eth u + s x))\xi(s + \eth y)
\end{equation}

Let $(e_1,\ldots,e_d)$ be the canonical basis for $\C^d$. Define $u_{p,q} e_j = \exp\left(\frac{2 i \pi p}{q}\right)$ and $u_{p,q} e_j = e_{ (j+1) \mod d }$ for all $j \in \{1,\ldots,d\}$.

For all $(n,m) \in \Z^2$, we then define:
\begin{equation*}
  \varpi_{d,p,q}^{n,m}\xi = \exp\left(\frac{i\pi p n m}{q}\right)u_{p,q}^n v_{p,q}^m \;\cdot\; \sigma_{p,q,d}^{n,m} \xi \text{.}
\end{equation*}
By construction, $\varpi_{d,p,q}$ is a unitary projective representation of $\Z^2$ for the $2$-cocycle $\cocycle{\theta}$, thus it defines a unique *-representation of $\qt{\theta}$ on $L^2(\R)\otimes\C^d = L^2(\R,\C^d)$ by integration.

Now, we note that the following space is carried into itself by $\varpi_{p,q,d}$: 
\begin{equation*}
  \mathcal{S}(\C^d) = \left\{ f : \R\rightarrow\C^d:\forall p \in \R[X] \; \lim_{x\rightarrow\pm\infty}p(x)f(x) =0 \right\}\text{.}
\end{equation*}

For any $\xi,\omega\in\mathcal{S}(\C^d)$, we set:
\begin{equation*}
\inner{\xi}{\omega}{\HeisenbergMod{\theta}{p,q,d}} : (n,m) \in \Z^2 \longmapsto \inner{\varpi_{p,q,\eth,d}^{n,m}\xi}{\omega}{L^2(\R)\otimes E} \text{.}
\end{equation*}
The completion of $\mathcal{S}(\C^d)$ for the norm induced by the $\qt{\theta}$-valued inner product $\inner{\cdot}{\cdot}{\HeisenbergMod{}{}}$ is a Hilbert module over $\qt{\theta}$ called the \emph{Heisenberg module} $\HeisenbergMod{\theta}{p,q,d}$.

We define the projective unitary representation $\sigma_{p,q,d}$ of $\R^2$ for the $2$-cocycle $\cocycle{\eth}$ by setting:
\begin{equation*}
  \forall x,y \in \R^2 \quad \sigma_{\eth,d}^{x,y} = \alpha_{\eth,d}^{x,y,\frac{xy}{2}} \text{.}
\end{equation*}
This representation is the Schr{\"o}dinger representation for the ``Planck constant'' $\eth$. We then proved in \cite{Latremoliere17a} that Heisenberg modules over quantum tori give rise to {\gQVB s}.

\begin{theorem}[{\cite[Theorem 6.11]{Latremoliere17a}}]\label{HeisenbergMod-dnorm-thm}
Let $\HeisenbergMod{\theta}{p,q,d}$ be the Heisenberg module over $\qt{\theta}$ for some $\theta\in\R$,  $p\in\Z$, $q \in \N\setminus\{0\}$ and $d\in q\N\setminus\{0\}$. Let $\eth = \theta - \frac{p}{q}$ and assume $\eth \not= 0$. Let $\|\cdot\|$ be a norm on $\R^2$. If we set, for all $\xi \in \HeisenbergMod{\theta}{p,q,d}$:
\begin{equation*}
\CDN_\theta^{p,d,q}(\xi) = \sup\left\{ \norm{\xi}{\HeisenbergMod{\theta}{p,q,d}}, \frac{\left\|\sigma_{\eth,d}^{x,y}\xi - \xi\right\|_{\HeisenbergMod{\theta}{p,q,d}}}{2\pi|\eth|\|(x,y)\|}  : (x,y) \in \R^2\setminus\{0\} \right\} \text{,}
\end{equation*}
and for all $a\in\qt{\theta}$:
\begin{equation*}
\Lip_\theta(a) = \sup\left\{\frac{\left\|\beta_\theta^{\exp(ix),\exp(iy)}a - a\right\|_{\qt{\theta}}}{\|(x,y)\|} : (x,y)\in\R^2\setminus\{0\} \right\}
\end{equation*}
then $\left(\HeisenbergMod{\theta}{p,q,d}, \inner{\cdot}{\cdot}{\HeisenbergMod{\theta}{p,q,d}}, \CDN_\theta^{p,d,q}, \qt{\theta}, \Lip_\theta\right)$ is a Leibniz {\gQVB}.
\end{theorem}

\emph{We will use the notation of Theorem (\ref{HeisenbergMod-dnorm-thm}) for the remainder of this paper.} In particular, we fix a norm $\norm{\cdot}{}$ on $\R^2$ for the remainder of this paper.

\medskip

We note that while in general, the action of the Heisenberg groups on Heisenberg modules is not by isometry of our D-norms, we can still find some related operators which act almost as isometries. This begins the new material for this paper.

\begin{notation}
  If $E$ is a metric space, $x\in E$ and $r > 0$, the closed ball of center $x$ and radius $r$ in $E$ is denoted by $E[x,r]$.
\end{notation}

\begin{lemma}\label{contraction-lemma}
  Let $p\in\Z$, $q\in \N$ and $d\in\N\setminus\{0\}$. For all $\varepsilon > 0$ there exists $\delta > 0$ such that, if $f : \R^2 \rightarrow \R_+$ is an integrable function supported on a $\R^2[0,\delta]$ such that $\iint_{\R^2} f \leq 1$, and if $\eth \not= 0$, then for all $\xi \in \mathcal{S}(\C^d)$, we have:
\begin{equation*}
\CDN_{\frac{p}{q} + \eth}^{p,q,d} \left( \sigma_{\eth,d}^{f} \xi \right) \leq (1 + \varepsilon) \CDN_{\frac{p}{q} + \eth}^{p,q,d}(\xi)\text{.}
\end{equation*}
\end{lemma}

\begin{proof}
Let $\eth \in \R\setminus\{0\}$ and let $\theta = \frac{p}{q} + \eth$. We first record that for all $(x,y,u), (z,w,v) \in \HeisenbergGroup$:
\begin{equation*}
\alpha_{\eth,d}^{x,y,u}\alpha_{\eth,d}^{z,w,v} = \exp\left( 2i\pi \eth (x w - z y) \right)\alpha_{\eth,d}^{z,w,v}\alpha_{\eth,d}^{x,y,u} \text{.}
\end{equation*}

Next, we denote by $\|\cdot\|_2$ the standard Euclidean norm on $\R^2$ and, since $\R^2$ is finite dimensional, we can find $k > 0$ such that $\|\cdot\|_2 \leq k \|\cdot\|$. For all $x,y,z,w \in \R$, we then compute:

\begin{align*}
\left| \exp\left( 2i\pi \eth (x w - z y) \right) - 1 \right| &= 2|\sin(\eth\pi(x w - z y))|
\leq 2\pi|\eth| |x w - y z| \\
&\leq 2\pi |\eth| \|(x,y)\|_2 \|(z,w)\|_2
\leq 2\pi k |\eth| \|(x,y)\| \|(z,w)\|_2 \text{.}
\end{align*}

Let $\varepsilon > 0$. Let $f$ be a nonnegative integrable function supported on $\R^2[0,\delta]$ with $\iint_{\R^2} f \leq 1$.

For all $(z,w) \in \R^2[0,\delta]$ and $(x,y) \in \R^2$ with $\|(x,y)\|\leq \delta = \frac{\varepsilon}{k}$, we compute:
\begin{align*}
&\, \left\| \alpha_{\eth,d}^{x,y,\frac{x y}{2}}\left(\iint_K  f(z,w)\alpha_{\eth,d}^{z,w,\frac{z w}{2}}\xi \, dx dy\right) - \iint f(z,w)\alpha_{\eth,d}^{z,w,\frac{z w}{2}}\xi \, dx dy \right\|_{\HeisenbergMod{\theta}{p,q,d}}\\
&\leq \iint_K f(z,w) \left| \exp\left( 2i\pi \eth (x y - z w) \right) - 1 \right| \left\| \alpha_{\eth,d}^{z,w,\frac{z w}{2}}\alpha_{\eth,d}^{x,y,\frac{x y}{2}}\xi  \right\|_{\HeisenbergMod{\theta}{p,q,d}} \, dz dw \\
&+ \iint_K f(z,w) \left\|\alpha_{\eth,d}^{z,w,\frac{ z w }{2}}\left(\alpha_{\eth,d}^{x,y,\frac{x y}{2}} \xi - \xi\right)\right\|_{\HeisenbergMod{\theta}{p,q,d}} \, dz dw \\
&\leq 2\pi k |\eth| \delta \|(x,y)\| \iint_K f \|\xi\|_{\HeisenbergMod{\theta}{p,q,d}} + \iint_K f 2\pi |\eth| \|(x,y)\| \CDN_{\theta}^{p,q,d}(\xi) \\
&\leq \left(k\delta + 1 \right) 2\pi|\eth| \|(x,y)\| \CDN_\theta^{p,q,d}(\xi) \\
&\leq (\varepsilon + 1)2\pi|\eth| \|(x,y)\| \CDN_\theta^{p,q,d}(\xi) \text{.}
\end{align*}

By definition (see Theorem (\ref{HeisenbergMod-dnorm-thm})), and since $\sigma_{\eth,d}$ acts by isometries for $\HeisenbergMod{\theta}{p,q,d}$, we now have:
\begin{equation*}
\begin{split}
\CDN(\sigma_{\eth,d}^{f}\xi) &= \sup\left\{ \norm{\sigma_{\eth,d}^{f}\xi}{\HeisenbergMod{\theta}{p,q,d}}, \frac{\left\|\alpha_{\eth,d}^{x,y,\frac{x y}{2}} \sigma_{\eth,d}^{f} \xi - \sigma_{\eth,d}^{f} \xi \right\|_{\HeisenbergMod{\theta}{p,q,d}}}{2\pi |\eth| \|(x,y)\|} : x,y\in\R\setminus\{0\} \right\} \\
&\leq (\varepsilon + 1) \CDN_\theta^{p,q,d}(\xi) \text{.}
\end{split}
\end{equation*}

Our lemma is now proven.
\end{proof}

\section{A continuous field of D-norms}

The essential observation which we use to prove our main result, Theorem (\ref{main-qt-thm}), is that the D-norms on Heisenberg modules form continuous families:

\begin{theorem}\label{d-norm-cont-thm}
Let $p\in\Z$, $q\in\N\setminus\{0\}$ and $d\in q\N\setminus\{0\}$. Let $\xi \in \mathcal{S}(\C^d)$. Let $\rho : \R\setminus\{\frac{p}{q}\} \rightarrow \R\setminus\{0\}$ be a continuous function. If $(\theta_k)_{k\in\N}$ is a sequence in $\R\setminus\left\{\frac{p}{q}\right\}$ converging to $\theta_\infty\not=\frac{p}{q}$  then:
\begin{equation*}
\lim_{k\rightarrow\infty} \CDN_{\theta_k}^{p,q,d}(\xi_{\rho(\theta_k)}) = \CDN_{\theta_\infty}^{p,q,d}(\xi_{\rho(\theta_\infty)}) \text{,}
\end{equation*}
where $\xi_{\rho(\theta)} : t \in \R \mapsto \xi(\rho(\theta) t)$ for all $\theta\not=\frac{p}{q}$.
\end{theorem}

To prove Theorem (\ref{d-norm-cont-thm}), our first step is to reformulate the expression of our D-norms.

\begin{lemma}\label{omega-cont-lemma}
Let $p\in\Z$, $q\in\N\setminus\{0\}$ and $d \in q\N\setminus\{0\}$. Let $r : \R\setminus\{0\} \rightarrow \R \setminus\{0\}$ be a continuous function.

If $\xi \in \mathcal{S}(\C^d)$, then for all $(x,y) \in \R^2$ with $\|(x,y)\| = 1$, $\eth \in\R\setminus\{0\}$, the function:
\begin{equation*}
t\in(0,\infty) \mapsto \omega_{x,y,t,\eth} = \frac{\sigma_{\eth,d}^{tx,ty}\xi_{r(\eth)} - \xi_{r(\eth)}}{2\pi\eth t}
\end{equation*}
where $\xi_{r(\eth)} : t\in\R\mapsto \xi(r(\eth)t)$, can be extended by continuity at $0$. Moreover, for all $\eth \in\R\setminus\{0\}$:
\begin{multline*}
  \CDN_{\eth + \frac{p}{q}}^{p,q,d} (\xi_{r(\eth)}) = \\
  \sup\left\{ \norm{\xi_{r(\eth)}}{\HeisenbergMod{\theta}{p,q,d}}, \left\|\omega_{x,y,t,\eth}\right\|_{\HeisenbergMod{\theta}{p,q,d}} : (x,y) \in \R^2,\|(x,y)\| = 1, t \in [0,1] \right\}
\end{multline*}
and
\begin{equation*}
(x,y,t,\eth) \in \R^2[0,1]\times\R\times(\R\setminus\{0\}) \mapsto \inner{\omega_{x,y,t,\eth}}{\omega_{x,y,t,\eth}}{\HeisenbergMod{\eth + \frac{p}{q}}{p,q,d}}
\end{equation*}
is continuous to $\left(\ell^1(\Z^2),\|\cdot\|_{\ell^1(\Z^2)}\right)$.
\end{lemma}

\begin{proof}
We begin by setting a domain over which we shall study our functions $\omega$. For our purpose, we choose some arbitrary $\eth_\infty \not= 0$ and then $0 < \eth_- < \eth_+$ such that $|\eth_\infty| \in (\eth_-,\eth_+)$. We set:
\begin{equation*}
\Omega = \left\{ (x,y,\eth) \in \R^3 : \|(x,y)\| = 1, |\eth| \in [\eth_-, \eth_+] \right\}
\end{equation*}
while:
\begin{equation*}
\Sigma = \{(x,y,t,\eth) \in \R^4 : (x,y,\eth) \in \Omega, t \in [0,1] \}
\end{equation*}
and:
\begin{equation*}
\Sigma_\ast = \{(x,y,t,\eth) \in \R^4 : (x,y,\eth) \in \Omega, t \in (0,1] \} \text{.}
\end{equation*}

Let $\xi \in \mathcal{S}(\C^d)$ and let $M_0 > 0$ be chosen so that for all $s \in \R$:
\begin{equation*}
\max\{\|\xi^{(n)}(s)\|_{\C^d}, \|s\xi^{(n)}(s)\|_{\C^d}  : n \in \{0,1,2,3,4\} \} \leq \frac{M_0}{1 + s^2} \text{.}
\end{equation*}

Now, $r$ is continuous on $[\eth_-,\eth_+]$, and thus there exists $R_- , R_+ > 0$ such that $R_- \leq r(\eth) \leq R_+$ for all $\eth \in [\eth_-, \eth_+]$. Thus for all $s\in\R$:
\begin{equation*}
\max\{\|\xi^{(n)}(r(\eth)s)\|_{\C^d}  : n \in \{0,1,2,3,4\} \} \leq \frac{M_0}{1 + R_-^2 s^2}
\end{equation*}
and
\begin{equation*}
\max\left\{ \|s\xi^{(n)}(r(\eth) s)\|_{\C^d} : n \in \{0,1,2,3,4\} \right\} \leq \frac{M_0}{R_-(1 + R_-^2 s^2)}\text{.}
\end{equation*}

We thus conclude that there exists $M > 0$ such that for all $\eth \in \R$ with $|\eth| \in [\eth_-,\eth_+]$ and for all $s\in\R$:
\begin{equation*}
\max\{\|\xi^{(n)}(r(\eth)s)\|_{\C^d}, \|s\xi^{(n)}(r(\eth)s)\|_{\C^d}  : n \in \{0,1,2,3,4\} \} \leq \frac{M}{1 + s^2} \text{.}
\end{equation*}

We first extend $(x,y,t,\eth)\in\Sigma_\ast\mapsto\omega_{x,y,t,\eth}$ to $\Sigma$ by continuity. We set, for all $(x,y)\in \R^2$ with $\|(x,y)\|=1$ and $\eth \in \R\setminus\{0\}$:
\begin{equation*}
\omega_{x,y,0,\eth} : s\in \R\mapsto \frac{i x s}{\eth} \xi(r(\eth)s) +  y r(\eth)\frac{\xi'(r(\eth)s)}{2\pi} \text{.}
\end{equation*}

Now, for all $(x,y,t,\eth) \in \Sigma_\ast$, we observe that for all $s\in\R$:

\begin{equation*}
\begin{split}
\omega_{x,y,t,\eth}(s) &= \frac{\exp(i\pi \left(t^2 \eth xy + 2 tx s \right))\xi_{r(\eth)}(s + \eth ty) - \xi_{r(\eth)}(s)}{2\pi\eth t}\\
&= \xi_{r(\eth)}(s + \eth yt) \frac{\exp(i\pi ( \eth t^2 xy + 2 tx s)) - 1}{2\pi\eth t} + \frac{\xi_{r(\eth)}(s + \eth ty) - \xi_{r(\eth)}(s)}{2\pi\eth t}\text{.}
\end{split}
\end{equation*}
Since $\xi$ is a Schwartz function, thus differentiable, we have for all $(x,y)\in\R^2$ with $\|(x,y)\|=1$ and $\eth \in \R\setminus\{0\}$:
\begin{equation*}
\lim_{t\rightarrow 0^+} \omega_{x,y,t,\eth}(s) = \omega_{x,y,0,\eth}(s) \text{.}
\end{equation*}

We observe that our statement thus far is about pointwise convergence of the family of functions $(\omega_{x,y,t,\eth})_{t > 0}$ to $\omega_{x,y,0,\eth}$ for fixed $x,y\in\R^2$ with $\|(x,y)\|=1$ and $\eth \not=0$. This is different from the notion of convergence in the $C^\ast$-Hilbert norm. To obtain convergence for the Heisenberg $C^\ast$-Hilbert norm, we now proceed to establish some regularity properties for $\omega$, in order to apply \cite[Lemma 3.2]{Latremoliere17a}.

By Proposition \cite[Proposition 4.8]{Latremoliere17a}, we already know that there exists $M_1 > 0$ such that for all $(x,y,t,\eth) \in \Sigma_\ast$ and $s\in\R$:
\begin{equation*}
\|\sigma_{\eth,d}^{tx,ty}\xi_{r(\eth)}(s) - \xi_{r(\eth)}(s)\|_{\C^d} \leq \frac{M_1 t\|(x,y,\frac{1}{2}t xy)\|_1}{1 + s^2}\text{,}
\end{equation*}
where $\|\cdot\|_1$ is the usual $1$-norm on $\R^3$.

The map $(x,y,t) \in \R^3 \mapsto M_1 t (x,y,\frac{t x y}{2})$ is continuous on the compact set $K = \{ (x,y,t) \in \R^2 : \|(x,y)\|=1, t\in[0,1]\}$, so there exists $M_2 > 0$ such that:
\begin{equation*}
\forall (x,y,t) \in K \quad \left\|M_1 t \left(x,y,\frac{ t x y }{2}\right)\right\|_1 \leq M_2\text{.}
\end{equation*}

Thus:
\begin{equation*}
\|\omega_{x,y,t,\eth}(s)\|_{\C^d} \leq \frac{t \|(x,y,\frac{1}{2}t xy)\|_1}{2\pi|\eth| t} \frac{M_1}{1+s^2} \leq \frac{\frac{M_2}{2\pi|\eth|}}{1 + s^2}\leq \frac{\frac{M_2}{2\pi|\eth|_-}}{1 + s^2}  \text{.}
\end{equation*}

On the other hand, by assumption and since $\{(x,y)\in\R^2:\norm{(x,y)}{}=1\}$ is compact, there exists $M_2'$ such that if $\norm{(x,y)}{} = 1$ then:
\begin{equation*}
\|\omega_{x,y,0,\eth}(s)\|_{\C^d} \leq |x|\frac{M}{|\eth|(1 + s^2)} + |y|\frac{R_+ M}{2\pi(1+s^2)} \leq \frac{M_2'}{1+s^2} \text{.}
\end{equation*}

In summary, there exists $M_3 = \max\left\{\frac{M_2}{2\pi\eth_-}, M_2' \right\} > 0$ such that for all $(x,y,\eth,t) \in \Sigma$ and all $s\in\R$:
\begin{equation*}
\left\|\omega_{x,y,t,\eth}(s) \right\|_{\C^d} \leq \frac{M_3}{1 + s^2} \text{.}
\end{equation*}

By construction, $(\omega_{x,y,t,\eth})_{t>0}$ converges pointwise to $\omega_{x,y,0,\eth}$ as $t\rightarrow 0$. We now prove that this convergence is indeed uniform. 

We begin with the following computation for all $(x,y,t,\eth) \in \Sigma_\ast$ and for all $s\in\R$:
\begin{align*}
&\, \|\omega_{x,y,t,\eth}(s) - \omega_{x,y,0,\eth}(s)\|_{\C^d} \\
&= \left| \frac{\sigma_{\eth,d}^{tx,ty}\xi_{r(\eth)}(s) - \xi_{r(\eth)}(s)}{2\pi t \eth } - \left(\frac{i x s}{\eth} \xi_{r(\eth)}(s) + y \frac{r(\eth)}{2\pi} \xi'(r(\eth)s)\right) \right| \\
&\leq \left| \left(\frac{\exp(i\pi ( \eth t^2 xy + 2s t x) ) - 1}{2\pi t \eth} - \frac{i x s}{\eth}\right) \xi_{r(\eth)}(s + \eth t y)\right| \\
&\quad + \left| \frac{i x s}{\eth}\left( \xi_{r(\eth)}(s + \eth t y) -  \xi_{r(\eth)}(s) \right) \right| \\
&\quad + \frac{1}{2\pi}\left| \frac{\xi_{r(\eth)}(s + \eth ty) - \xi_{r(\eth)}(s)}{\eth t} - y r(\eth) \xi'(r(\eth)s)\right|\text{.}
\end{align*}

We now bound each of the three terms in the above expression. We will obtain constants below, denoted by $M_4,M_5,\ldots$ that are uniform in $x,y,\eth,t$ and $s$. We first note that for all $(x,y,t,\eth)\in\Sigma_\ast$, by the mean value theorem, if $g:t \mapsto \exp(i\pi(t^2 \eth xy + 2 tx s))$, then there exists $t_c \in [0,t]$ such that:
\begin{align*}
\left| \frac{\exp(i\pi(\eth t^2 xy + 2 tx s)) - 1}{2\pi t \eth} - x \frac{i s}{\eth} \right|
&\leq \left| \frac{g(t) - g(0)}{2\pi t \eth} - g'(0) \right| = \left|\frac{t}{4\pi\eth} g''(t_c)  \right| \\
&= \frac{|t|}{2\pi\eth}\left| (i \pi \eth x y - 2 \pi^2 (t_c \eth x y + x s)^2)\exp(2i\pi s t_c x)\right|\\
&= \frac{|t|\left|i \pi \eth x y - 2 \pi^2 (t_c \eth x y + x s)^2\right|}{2\pi\eth} \text{.}
\end{align*}

Once again, by continuity over the compact $\Omega\times[0,1]$, there exists $M_4 > 0$ such that for all $(x,y,\eth,t) \in \Omega\times[0,1]$ and for all $s\in \R$, we obtain (after expanding the square and by the triangle inequality):
\begin{equation*}
 \frac{\left|i \pi \eth x y - 2 \pi^2 (t_c \eth x y + x s)^2\right|}{2\pi\eth} \leq M_4(1+|s|)^2 \text{.}
\end{equation*}
Therefore, for all $s\in \R$ and $(x,y,\eth,t)\in\Omega\times[0,1]$:
\begin{equation*}
  \left| \left(\frac{\exp(i\pi ( \eth t^2 xy + 2s t x) ) - 1}{2\pi t \eth} - x \frac{i s}{\eth}\right) \xi_{r(\eth)}(s + \eth t y)\right| \\
  \leq |t| \frac{M_4 M (1 + |s|)^2}{1+(s+\eth t y)^2} \text{.}
\end{equation*}

Again by compactness, $(x,y,\eth,t)\in\Omega\times[0,1]\mapsto \eth t y$ is bounded  by some $M_4'$. If $s > M_4'$ then $(s+\eth t y)^2 \geq (s-M_4')^2$ while if $s<-M_4'$ then $(s+\eth t y)^2 \geq (s+M_4')^2$. Thus, there exists $M_5 > 0$ such that if $|s| > M_4'$ then:
\begin{equation*}
\frac{M M_4 \left(1 + |s|\right)^2}{\left(1 + (s+\eth t y)^2\right)} \leq \max\left\{ \frac{M M_4 \left(1 + |s|\right)^2}{\left(1 + (s-M_4')^2\right)},\frac{M M_4 \left(1 + |s|\right)^2}{\left(1 + (s+M_4')^2\right)} \right\} \leq M_5 \text{.}
\end{equation*}

The function $(x,y,t,s)\in\R^4\mapsto \frac{M M_4 \left(1 + |s|\right)^2}{1 + (s+\eth t y)^2}$ is continuous, and thus it is bounded by some $M_6 > 0$ on the compact $\Sigma\times[-M_4',M_4']$. Thus there exists $M_7 = \max\{M_5,M_6\} > 0$ such that for all $(x,y,\eth,t) \in \Sigma_\ast$ and $s\in\R$, we have:
\begin{equation*}
\left| \left(\frac{\exp(2i\pi tx s) - 1}{\eth t} - \frac{2i\pi x}{\eth}\right) \xi_{r(\eth)}(s + t \eth y)\right| \leq  M_7|t|\text{.}
\end{equation*}

Consequently, if $|t| < \delta_1 = \frac{\varepsilon}{3 M_7}$, then:
\begin{equation}
\left| \left(\frac{\exp(i\pi (t^2 \eth x y + 2 tx s)) - 1}{\eth t} - \frac{2i\pi x}{\eth}\right) \xi_{r(\eth)}(s + t \eth y)\right| < \frac{\varepsilon}{3} \text{.}
\end{equation}

We now turn to the second term we wish to bound. As we noted before, by compactness, there exists $M_8 > 0$ such that if $\norm{(x,y)}{}=1$ and $|\eth|\in[\eth_-,\eth_+]$ then $\frac{|x|}{\eth} \leq M_8$ and $|\eth y| \leq M_8$. On the other hand, the function $s\in\R\mapsto s\xi(s)$ is uniformly continuous on $\R$ as a Schwartz function. There exists $\delta_2 > 0$ such that for all $0 < t < \delta_2$ and for all $s\in\R$, we have $|(s+t)\xi(s + t) - \xi(s)| < \frac{\varepsilon}{6 M_8}$. 

Moreover, since $\xi$ is bounded on $\R$, we may choose $\delta_2 > 0$ small enough so that $\delta_2\sup_{s\in\R}\|\xi(s)\|_{\C^d} < \frac{\varepsilon}{6 M_8}$. 

Let now $\delta_3 = \frac{\delta_2}{M_8 R_+}$. If $|t| < \delta_3$ then $|r(\eth) \eth t y| < \delta_2$ and therefore:
\begin{align*}
&\, \left\|x\frac{i s}{\eth}\xi_{r(\eth)}(s + \eth t y) - x\frac{is}{\eth}\xi_{r(\eth)}(s)\right\|_{\C^d}\\
&\quad\quad \leq M_8\left( \|(r(\eth)s + r(\eth) \eth t y)\xi(r(\eth)s + r(\eth)\eth t y) - r(\eth)s\xi(r(\eth) s)\|_{\C^d} \right. \\
&\quad\quad\quad \left. + |r(\eth) \eth t y|\|\xi(r(\eth)(s + \eth t y))\|_{\C^d}\right) \\
&\quad\quad \leq M_8 \left(\frac{\varepsilon}{6 M_8} + \delta_2\sup_{s\in\R}\|\xi(s)\|_{\C^d}\right) \leq \frac{\varepsilon}{6} + \frac{\varepsilon}{6} = \frac{\varepsilon}{3}\text{.}
\end{align*}

We thus deduce that for all $(x,y,t,\eth) \in \Sigma$, if $|t| < \delta_3$, then:
\begin{equation}
\sup_{s\in\R} |x\frac{i s}{\eth}\xi_{r(\eth)}(s + \eth t y) - x\frac{is}{\eth}\xi_{r(\eth)}(s)| < \frac{\varepsilon}{3}\text{.}
\end{equation}

Last, since $\xi'$ is also a Schwartz function and in particular, also uniformly continuous on $\R$, there exists $\delta_4 > 0$ such that $|\xi'(s + r) - \xi'(s)| < \frac{2\pi\eth_-\varepsilon}{3 R_+ M_8}$ for all $0\leq r  < \delta_4$. Thus for all $(x,y,t,\eth) \in \Sigma$ and $s\in\R$, if $|t| < \delta_5 = \frac{\delta_4}{R_+ M_8}$ and $\norm{(x,y)}{}=1$, $s\in\R$, $t\in[0,1]$ and $|\eth| \in [\eth_-,\eth_+]$ then (since $|y|\leq \frac{M_8}{\eth_-}$):
\begin{multline*}
\left| \frac{\xi_{r(\eth)}(s + t\eth y) - \xi_{r(\eth)}(s)}{\eth} - y t r(\eth) \xi'(r(\eth)s)\right| \\ \leq |r(\eth)| \int_{0}^{t} |y||\xi'(r(\eth)(s + r\eth y)) - \xi'(r(\eth)s)| \,dr \leq t \frac{2\pi\varepsilon}{3} \text{.}
\end{multline*}

Thus for all $(x,y,t,\eth)\in\Sigma_\ast$ with $|t|<\delta_4$, we have:
\begin{equation*}
\sup_{s\in\R} \frac{1}{2\pi}\left| \frac{\xi_{r(\eth)}(s + \eth t y) - \xi_{r(\eth)}(s)}{\eth t} - y r(\eth) \xi'(r(\eth)s)\right| < \frac{\varepsilon}{3} \text{.}
\end{equation*}

In conclusion, for all $(x,y,t,\eth) \in \Sigma_\ast$ with $0 < t < \min\{\delta_1,\delta_3,\delta_5\}$ and for all $s\in\R$, we have established:
\begin{equation*}
\sup_{s\in\R} |\omega_{x,y,t,\eth}(s) - \omega_{x,y,0,\eth}(s)| < \varepsilon \text{.}
\end{equation*}

In other words, setting for all $t \in (0,1]$:
\begin{equation*}
\mathsf{f}_t : (x,y,\eth,s) \in \Omega \times\R \mapsto \omega_{x,y,t,\eth}(s)
\end{equation*}
the function $t\in(0,1]\mapsto f_t$ converges uniformly on $\Omega\times \R$ to:
\begin{equation*}
\mathsf{f}_0 : (x,y,\eth,s) \in \Omega\times\R \mapsto \omega_{x,y,0,\eth}(s)
\end{equation*}
when $t$ goes to $0$.

Since $(x,y,t,\eth,s) \in \Sigma_\ast\times\R \mapsto \mathsf{f}_t(x,y,\eth,s)$ and $\mathsf{f}_0$ are both continuous, we deduce, in particular, that:
\begin{equation*}
(x,y,t,\eth,s) \in \Sigma\times\R \mapsto \omega_{x,y,t,\eth}(s)
\end{equation*}
is (jointly) continuous. 

\worknote{Indeed if $(x_n,y_n,t_n,\eth_n,s_n)_{n\in\N} \in \left(\alpha\times\R\right)^\N$ is a sequence converging to $(x_\infty,y_\infty,t_\infty,\eth_\infty,s_\infty)$, then $(x_\infty,y_\infty,t_\infty,\eth_\infty,s_\infty) \in \alpha\times\R$ and:
\begin{itemize}
\item if $t_\infty > 0$ then $(x_\infty,y_\infty,t_\infty,\eth_\infty,s_\infty) \in \alpha_\ast\times \R$, then we simply use the continuity of $\mathsf{h}$,
\item if $t_\infty = 0$ then:
\begin{equation*}
\begin{split}
\|f_{t_n}(x_n,y_n,\eth_n,s_n) - f_0(x_\infty,y_\infty,\eth_\infty,s_\infty)\| &\leq
\|f_{t_n}(x_n,y_n,\eth_n,s_n) - f_0(x_n,y_n,\eth_n,s_n)\| \\
&\quad + \|f_{t_\infty}(x_n,y_n,\eth_n,s_n) - f_0(x_\infty,y_\infty,\eth_\infty,s_\infty)\| \\
&\leq \sup_{(x,y,\eth,s) \in \Omega\times\R}\|f_{t_n}(x,y,\eth,s) - f_0(x,y,\eth,s)\| \\
&\quad + \|f_{t_\infty}(x_n,y_n,\eth_n,s_n) - f_0(x_\infty,y_\infty,\eth_\infty,s_\infty)\|\text{.}
\end{split}
\end{equation*}
The first term on the right hand side converges to $0$ by uniform convergence of $(f_t)_{t \in (0,K]}$ to $f_0$ on $\Omega\times\R$; the second term converges to $0$ by continuity of $f_0$.
\end{itemize}
}

The entire reasoning up to now may be applied equally well to $\xi^{(n)}$ for $n\in\{1,2\}$ --- as one may check that $\omega^{(n)}$ is indeed obtained by substituting $\xi$ with $\xi^{(n)}$. 

Therefore, we are now able to apply \cite[Lemma 3.2]{Latremoliere17a} to conclude that:
\begin{equation*}
(x,y,t,\theta) \in \Sigma \mapsto \inner{\omega_{x,y,t,\eth}}{\omega_{x,y,t,\eth}}{\HeisenbergMod{\theta}{p,q,d}} \in \left(\ell^1(\Z^2),\|\cdot\|_{\ell^1(\Z^2)}\right)
\end{equation*}
is continuous as desired (to make notations clear: we pick a sequence $(\theta_n)_{n\in\N}$ converging to some $\theta$, and we choose $(x_n)_{n\in\N}\in\R^\N$, $(y_n)_{n\in\N}\in\R^\N$, and $t_n\in [0,1]^\N$ such that for all $n\in\N$, we have $(x_n,y_n,t_n,\theta_n-\frac{p}{q}) \in \Sigma$, and then we set, in the notations of \cite[Lemma 3.2]{Latremoliere17a}, the functions $\xi_n = f_{t_n}(x_n,y_n,\eth_n,\cdot)$ and $\xi_\infty = f_{t_\infty}(x_\infty,y_\infty,\eth_\infty,\cdot)$).

We conclude our proof by observing that by Theorem (\ref{HeisenbergMod-dnorm-thm}), $\CDN_\theta^{p,q,d}(\xi_{r(\eth)})$ is given by:
\begin{align*}
&\,\sup\left\{ \norm{\xi_{r(\eth)}}{\HeisenbergMod{\theta}{p,q,d}}, \frac{\left\| \sigma_{\eth,d}^{x,y}\xi_{r(\eth)} - \xi_{r(\eth)} \right\|_{\HeisenbergMod{\theta}{p,q,d}}}{2\pi|\eth| \|(x,y)\| } : (x,y)\in\R^2, \|(x,y)\|\leq 1\right\}\\
&= \sup\left\{ \norm{\xi_{r(\eth)}}{\HeisenbergMod{\theta}{p,q,d}}, \left\| \omega_{x,y,t,\eth} \right\|_{\HeisenbergMod{\theta}{p,q,d}} : \|(x,y)\| = 1, t \in [0,1] \right\}
\end{align*}
as stated.
\end{proof}

We now prove that D-norms on Heisenberg modules form continuous fields.

\begin{proof}[Proof of Theorem (\ref{d-norm-cont-thm})]
The result is trivial if $\xi = 0$, which is equivalent to $\CDN_{\vartheta}^{p,q,d}(\xi) = 0$ for all $\vartheta \in \R$ with $\vartheta\not = \frac{p}{q}$.

Now, fix $\theta \in \R\setminus\{\frac{p}{q}\}$ and set $\eth = \theta-\frac{p}{q}$. We shall prove that $\vartheta\mapsto\CDN_\vartheta^{p,q,d}(\xi_{\rho(\vartheta)})$ is continuous at $\theta$.

Let $\delta_1 > 0$ such that $I = [\eth-\delta_1,\eth+\delta_1] \subseteq \R\setminus\left\{0\right\}$. Let $r : \hslash \in I \rightarrow \rho\left(\hslash+\frac{p}{q}\right)$.

Let:
\begin{equation*}
\Upsilon = \left\{ (x,y,t)\in\R^3 : \|(x,y)\| = 1, t \in [0, 1] \right\}\text{.}
\end{equation*}

Let $\xi \in \mathcal{S}(\C^d)$. We use the notations of Lemma (\ref{omega-cont-lemma}) and we set:
\begin{equation*}
  \nu : (x,y,t,\hslash) \in \Upsilon\times I \mapsto \inner{\omega_{x,y,t,\hslash}}{\omega_{x,y,t,\hslash}}{\HeisenbergMod{\vartheta+\frac{p}{q}}{p,q,d}}
\end{equation*}

Let $\hslash\in I$ and let $\vartheta=\hslash + \frac{p}{q}$. By Lemma (\ref{omega-cont-lemma}), $\CDN_{\vartheta}^{p,q,d}(\xi_{r(\hslash)})$ is the maximum of $\norm{\xi_{r(\hslash)}}{\HeisenbergMod{\vartheta}{p,q,d}}$ and the square root of:
\begin{align*}
\mathrm{E}_{\hslash}^{p,q,d}(\xi_{r(\hslash)})^2
&= \sup\left\{\left\|\upsilon(x,y,t,\hslash)\right\|_{\qt{\vartheta}} : (x,y,t) \in \Upsilon \right\}\\
&\leq \sup\left\{\left\|\upsilon(x,y,t,\hslash)\right\|_{\ell^1(\Z^2)} : (x,y,t) \in \Upsilon \right\}
\end{align*}
where $\Upsilon$ is a compact subset of $\R^3$, independent of $\hslash$. By \cite[Proposition 3.5]{Latremoliere17a}, the function $\hslash \in I \mapsto \norm{\xi_{r(\hslash)}}{\HeisenbergMod{\vartheta}{p,q,d}}$ is continuous, so it is sufficient to show that $\hslash\in I \mapsto \mathrm{E}_\hslash^{p,q,d}(\xi_{r(\hslash)})$.

Now, since $\nu$ is continuous in $\left(\ell^1(\Z^2),\|\cdot\|_{\ell^1(\Z^2)}\right)$ by Lemma (\ref{omega-cont-lemma}), it is uniformly continuous on the compact $\Upsilon_2 = \Upsilon\times I$. 

Let $\|(z,w,s,\hslash)\|_\infty = \max\{|z|,|w|,|s|,|\hslash|\}$ for all $(z,w,s,\hslash) \in \R^4$.

Let $\delta_2 > 0$ be chosen so that for all $(x,y,t,s),(z,w,u,v) \in \Upsilon_2$ with $\|(x,y,t,s) - (z,w,u,v)\|_\infty < \delta_2$ we have:
\begin{equation*}
|\nu(x,y,t,s) - \nu(z,w,u,v)| < \frac{\varepsilon}{4}\text{.}
\end{equation*}

Let $G \subseteq \Upsilon_2$ be a $\delta_2$-dense finite subset of $\Upsilon_2$ in the sense of the norm $\|\cdot\|_\infty$. Let:
\begin{equation*}
F = \left\{ (z,w,r) \in \R^3 : \exists \hslash \in \R \quad (z,w,r,\hslash) \in G \right\} \text{.}
\end{equation*}
By construction, $F$ is finite and $\delta_2$-dense in $\Upsilon$ for the restriction of $\|\cdot\|_\infty$ to $\R^3 \sim \R^3\times\{0\}$.

Fix any $\hslash \in I$ and set $\vartheta = \hslash + \frac{p}{q}$. Now, let $(x,y,t) \in \Upsilon$. There exists $(z,w,s)\in F$ with $\max\{|x-z|,|y-w|, |t-s| \} < \delta_2$. We then observe:

\begin{align*}
  \norm{\upsilon(x,y,t,\hslash)}{\qt{\vartheta}}
  &\leq \left|\norm{\upsilon(x,y,t,\hslash)}{\qt{\vartheta}} - \norm{\upsilon(z,w,s,\hslash)}{\qt{\vartheta}}\right| + \norm{\upsilon(z,w,s,\hslash)}{\qt{\vartheta}} \\
  &\leq \norm{\upsilon(x,y,t,\hslash)-\upsilon(z,w,s,\hslash)}{\qt{\vartheta}} + \norm{\upsilon(z,w,s,\hslash)}{\qt{\vartheta}} \\
  &\leq \norm{\upsilon(x,y,t,\hslash)-\upsilon(z,w,s,\hslash)}{\ell^1(\Z^2)} + \norm{\upsilon(z,w,s,\hslash)}{\qt{\vartheta}} \\
  &\leq \frac{\varepsilon}{4} + \norm{\upsilon(z,w,s,\hslash)}{\qt{\vartheta}}\text{.}
\end{align*}

Let $\mathsf{F}_\hslash^{p,q,d}(\eta)$ be given by:
\begin{equation*}
\mathsf{F}_\hslash^{p,q,d}(\eta) = \max\left\{ \norm{\omega_{z,w,s,\hslash}}{\qt{\hslash+\frac{p}{q}}} : (z,w,s) \in F \right\}
\end{equation*}
for all $\eta \in \mathcal{S}(\C^d)$.

We thus have proven that for all $\hslash \in I$, the following holds:
\begin{equation*}
\mathsf{F}_\hslash^{p,q,d}(\xi_{r(\hslash)})^2 \leq \mathsf{E}_\hslash^{p,q,d}(\xi_{r(\hslash)})^2 \leq \frac{\varepsilon}{4} + \mathsf{F}_\hslash^{p,q,d}(\xi_{r(\hslash)})^2\text{.}
\end{equation*}

Therefore:
\begin{equation}\label{d-norm-cont-eq1}
\begin{split}
\left|\mathsf{E}_\hslash^{p,q,d}(\xi_{r(\hslash)})^2 - \mathsf{E}_\eth^{p,q,d}(\xi_{r(\eth)})^2\right| &\leq \frac{\varepsilon}{2} + |\mathsf{F}_\hslash^{p,q,d}(\xi_{r(\hslash)})^2 - \mathsf{F}_\eth^{p,q,d}(\xi_{r(\eth)})^2| \text{.}
\end{split}
\end{equation}

Note that for any $\eta\in\mathcal{S}(\C^d)$, the quantity $\mathsf{F}_\hslash^{p,q,d}(\eta)$ is finite as the maximum of finitely many values. Also note that the set $F$ does not change with $\hslash \in I$ --- the only dependence of $\mathsf{F}_\hslash^{p,q,d}$ on $\hslash$ is via the choice of the quantum torus norm $\|\cdot\|_{\HeisenbergMod{\hslash+\frac{p}{q}}{p,q,d}}$. 

Now the key observation is that $\hslash \in I \mapsto \mathsf{F}_\hslash^{p,q,d}(\xi)$ is continuous. Fix $(z,w,s) \in F$. By \cite[Proposition 3.5]{Latremoliere17a}, the function:
\begin{equation*}
\hslash \in I \mapsto \left\| \frac{\sigma_{\eth,d}^{sz,sw}\xi_{r(\hslash)} - \xi_{r(\hslash)}}{2\pi|\hslash| s} \right\|_{\HeisenbergMod{\hslash+\frac{p}{q}}{p,q,d}}
\end{equation*}
is continuous (we note that $\hslash\in I \mapsto \sigma_{\hslash,d}^{s z,s w}\xi_{r(\hslash)} - \xi_{r(\hslash)}$ satisfies the necessary hypothesis, owing to $\xi$ being a Schwartz function and $r$ being continuous. The details follow similar lines to our proof of Lemma (\ref{omega-cont-lemma}) and we shall omit them this time around).

Thus $\hslash\in I \mapsto \mathsf{F}_\hslash^{p,q,d}(\xi_{r(\hslash)})$ is the maximum of finitely many continuous functions, and is therefore continuous as well.

Thus there exists $\delta_3 > 0$ such that for all $\hslash \in [\eth-\delta_3,\eth+\delta_3]$ we have:
\begin{equation*}
|\mathsf{F}_\hslash^{p,q,d}(\xi_{r(\hslash)})^2 - \mathsf{F}_\eth^{p,q,d}(\xi_{r(\eth)})^2| < \frac{\varepsilon}{2}\text{.}
\end{equation*} 

Thus if $\delta = \min\{\delta_1,\delta_3\} > 0$ then for all $\hslash\in[\eth-\delta,\eth+\delta]$ we have:
\begin{equation*}
\left|\mathsf{E}_\hslash^{p,q,d}(\xi_{r(\hslash)})^2 - \mathsf{E}_\eth^{p,q,d}(\xi_{r(\eth)})^2\right| < \varepsilon\text{.}
\end{equation*}
Since $\mathsf{E}_\hslash(\xi_{r(\hslash)}) \geq 0$ for all $\hslash\in[\eth - \delta,\eth+\delta]$ and $\sqrt[2]{\cdot}$ is a continuous function on $[0,\infty)$, we have shown that:
\begin{equation*}
\hslash \in I \mapsto \mathsf{E}_{\hslash}^{p,q,d}(\xi_{r(\hslash)})
\end{equation*}
is continuous at $\eth$. Therefore, as the maximum as two continuous functions by Lemma (\ref{omega-cont-lemma}) and \cite[Proposition 3.5]{Latremoliere17a}, the function $\vartheta \mapsto \CDN_{\vartheta}^{p,q,d}(\xi_{\rho(\vartheta)})$ is continuous at $\theta$ as well.
\end{proof}

\section{Convergence}

We now present our main convergence result for the modular propinquity:
We now conclude our paper with the main result of its second part, which demonstrates that the modular propinquity endows the moduli space of Heisenberg modules over quantum $2$-tori with a nontrivial geometry.

\begin{theorem}\label{main-qt-thm}
Let $\|\cdot\|$ be a norm on $\R^2$. For all $\theta \in \R$, we equip the quantum torus $\qt{\theta}$ with the L-seminorm:
\begin{equation*}
\Lip_\theta : a \in \sa{\A} \mapsto \sup\left\{ \frac{\left\| \beta_\theta^{\exp(i x),\exp(i y)}a - a\right\|_{\qt{\theta}}}{\|(x,y)\|} : (x,y) \in \R^2\setminus\{0\} \right\} \text{.}
\end{equation*}

For all $p\in\Z$, $q\in\N\setminus\{0\}$ and $d\in q\N\setminus\{0\}$ and for all $\theta\in \R\setminus\left\{\frac{p}{q}\right\}$, we endow the Heisenberg module $\HeisenbergMod{\theta}{p,q,d}$ with the D-norm:
\begin{multline*}
\CDN_\theta^{p,q,d} : \xi \in \HeisenbergMod{\theta}{p,q,d} \mapsto \\
\sup\left\{ \|\xi\|_{\HeisenbergMod{\theta}{p,q,d}}, \frac{\left\|\sigma_{\eth,d}^{x,y}\xi - \xi \right\|_{\HeisenbergMod{\theta}{p,q,d}}}{2\pi\left(\theta - \frac{p}{q}\right) \|(x,y)\| } : (x,y) \in \R^2\setminus\{0\} \right\} \text{.}
\end{multline*}

Let $p \in \Z$ and $q \in \N\setminus\{0\}$. Let $d \in q\N\setminus\{0\}$. For any $\theta \in \R\setminus\left\{\frac{p}{q}\right\}$, we have:
\begin{multline*}
\lim_{\vartheta\rightarrow\theta} \modpropinquity{}\left(\left(\HeisenbergMod{\vartheta}{p,q,d}, \inner{\cdot}{\cdot}{\HeisenbergMod{\vartheta}{p,q,d}}, \CDN_\vartheta^{p,q,d}, \qt{\vartheta}, \Lip_{\vartheta} \right),\right. \\
\left. \left(\HeisenbergMod{\theta}{p,q,d}, \inner{\cdot}{\cdot}{\HeisenbergMod{\theta}{p,q,d}}, \CDN_\theta^{p,q,d}, \qt{\theta}, \Lip_{\theta}\right) \right) = 0\text{.}
\end{multline*}
\end{theorem}

To prove Theorem (\ref{main-qt-thm}), we begin with a few lemmas. Our first step consists in finding an appropriate choice of anchors. We establish two lemmas to this end. The first lemma extends \cite[Lemma 6.9]{Latremoliere17a} by proving that while the range of the operators involved in \cite[Lemma 6.9]{Latremoliere17a} depends on the parameters used to define the Heisenberg modules, its dimension does not. The second lemma then uses the particular basis of Hermite functions obtained in the first lemma to construct our anchors.

\begin{lemma}\label{same-dim-approx-lemma}
For all $j\in\N$ and $\eth > 0$, let:
\begin{equation*}
\psi_\eth^j : r \in [0,\infty) \mapsto \eth \exp\left(-\frac{\pi \eth r^2}{2}\right) \mathrm{L}_j\left(\pi \eth r^2\right)
\end{equation*}
where $\mathrm{L}_j : t \in \R\mapsto \sum_{k = 0}^n \frac{(-1)^k}{k!} {j \choose j - k} t^k$ is the $j^{\mathrm{th}}$ Laguerre polynomial. 

Let $f\in\C_0(\R)$ be compactly supported. For all $j\in\N$ and $\eth\not= 0$, we set:
\begin{equation*}
\mathrm{C}_\eth^j(f) = \sum_{k=0}^j \frac{j + 1 - k}{j + 1} \inner{f \psi_\eth^k}{\psi_\eth^k}{L^2(\R, r\,dr)}\psi_\eth^k \text{.}
\end{equation*}

For all $\varepsilon > 0$ and $\eth_0 \not=0$, there exists $N\in \N$ and $\delta \in (0,|\eth_0|)$ such that, for all $\eth \in [\eth_0 - \delta, \eth_0 + \delta]$, we have:
\begin{equation*}
\left\|f - \sum_{j=0}^N \mathrm{C}_\eth^j(f) \right\|_{L^1(\R_+, r dr)} \leq \varepsilon \text{.}
\end{equation*}
\end{lemma}

\begin{proof}

  Fix $\eth_0 > 0$. By \cite[Theorem 6.2.1]{Thangavelu93}, as in the proof of \cite[Lemma 6.9]{Latremoliere17a}, there exists $N > 0$ such that:
\begin{equation*}
\left\| f - \sum_{j=0}^N \mathrm{C}_{\eth_0}^j(f) \right\|_{L^1(\R_+, r dr)} \leq \frac{\varepsilon}{2} \text{.}
\end{equation*}

  Now, note that $\psi_\eth^j (t) = \eth\psi_1^j(\sqrt{\eth}t)$ for all $t \geq 0$, all $\eth>0$ and all $j\in\N$, with $\psi_1^j$ the $j^{\th}$-Laguerre function.

  Thus, for any $\eth > 0$ and for all $x\in [0,\infty)$, we have $\lim_{\eth\rightarrow\eth_0}\psi_{\eth}^j(x) = \psi_{\eth_0}^j(x)$.

  Moreover, since $\lim_{r\rightarrow\infty}\exp\left(\frac{r^2}{4}\right)\psi_1^j(r) = 0$ for all $j\in\N$, we conclude that there exists $K>0$ such that if $\eth \in \left[\frac{\eth_0}{2},\frac{3\eth_0}{2}\right]$ and $j\in \{0,\ldots,N\}$, then $x\geq K \implies \psi_\eth^j(x)\leq \exp\left(-\frac{\eth_0^2 r^2}{8}\right)$. 

  Since $(\eth,x)\in\left[\frac{\eth_0}{2},\frac{3\eth_0}{2}\right]\times[0,K]\mapsto \psi_{\eth}^j(x)$ is continuous on a compact, it is bounded; we thus can find $M > 0$ such that for all $j\in\{0,\ldots,N\}$, $\eth \in \left[\frac{\eth_0}{2},\frac{3\eth_0}{2}\right]$ and $x\in [0,K]$ we have $|\psi_\eth^j(x)|\leq M$.

  Therefore, if we set:
  \begin{equation*}
    g : x \in [0,\infty) \mapsto\begin{cases} M \text{ if $x\leq K$,} \\ \exp\left(-\frac{\eth_0^2 r^2}{8}\right) \text{ otherwise}\end{cases}
  \end{equation*}
  then for all $j\in\{0,\ldots,N\}$, $\eth \in \left[\frac{\eth_0}{2},\frac{3\eth_0}{2}\right]$ and $x\in [0,\infty)$, we have $|\psi_\eth^j(x)| \leq g(x)$. Of course, $g \in L^1(\R_+,r\,dr)$ and $g^2 \in L^1(\R_+,r\, dr)$.

  By Lebesgue Dominated Convergence Theorem, we thus conclude (as $f$ is bounded) that:
  \begin{equation*}
    \lim_{\eth\rightarrow\eth_0} \inner{f \psi_\eth^j}{\psi_\eth^j}{L^2(\R_+,r\,dr)} = \inner{f\psi_{\eth_0}}{\psi_{\eth_0}}{L^2(\R_+,r\,dr)}
  \end{equation*}
  and, by the same theorem again and by the triangle inequality:
  \begin{equation*}
    \lim_{\eth\rightarrow\eth_0} \norm{\mathrm{C}_\eth^j(f)-\mathrm{C}_{\eth_0}^j(f)}{L^1(\R_+,r\,dr)} = 0 \text{.}
  \end{equation*}

  Therefore, there exists $\delta > 0$ such that if $\eth \in [\eth_0-\delta,\eth_0+\delta]$, then:
  \begin{align*}
    \norm{f - \sum_{j=0}^N \mathrm{C}_{\eth}^j(f)}{L^1(\R_+,r\,dr)}
    &\leq
      \norm{f - \sum_{j=0}^N \mathrm{C}_{\eth_0}(f)}{L^1(\R_+,r\,dr)} \\
    &\quad + \sum_{j=0}^N\norm{\mathrm{C}_\eth^j(f)-\mathrm{C}_{\eth_0}^j(f)}{L^1(\R_+,r\,dr)} \leq \varepsilon
  \end{align*}
  as desired.
\end{proof}
  
We are now in a position to prove the existence of a good choice of anchors, which will be chosen as fixed elements in the space $\mathcal{S}(\C^d)$ of Schwartz functions which will give good approximations in norms in a whole family of Heisenberg modules.

\begin{lemma}\label{anchors-lemma}
Let $p \in \Z$, $q\in\N\setminus\{0\}$ and $d\in q\N\setminus\{0\}$. Let $\eth_0 \not= 0$. For all $\varepsilon > 0$, there exist $\delta \in \left(0,\frac{|\eth_0|}{2}\right)$ and a finite subset $\alg{F}$ of the closed unit ball of $\CDN_{\eth_0 + \frac{p}{q}}^{p,q,d}\setminus\{0\}$ for $\|\cdot\|_{\HeisenbergMod{\frac{p}{q} + \eth_0}{p,q,d}}$ such that, for all $\eth \in [\eth_0 - \delta, \eth_0 + \delta]$, the set:
\begin{equation*}
\left\{ \frac{\CDN_{\frac{p}{q} + \eth_0}^{p,q,d}(\omega)}{\CDN_{\frac{p}{q} + \eth}^{p,q,d}(\omega)}\omega : \omega \in \alg{F}  \right\}
\end{equation*}
is a $\varepsilon$-dense subset of the closed unit ball of $\|\cdot\|_{\HeisenbergMod{\frac{p}{q} + \eth}{p,q,d}}$.
\end{lemma}

\begin{proof}
We now note that thanks to a change of variable in the definition of the operator $\sigma_{\eth,d}^g$, it is sufficient to prove our result for $\eth > 0$. We shall henceforth assume $\eth > 0$.

We proceed in two steps. First, we reduce the problem to a finite dimensional question. We then show the desired continuity result.

Let $\varepsilon > 0$ be given. By \cite[Lemma 6.8,Lemma 3.8]{Latremoliere17a} and by Lemma (\ref{contraction-lemma}), there exists a function $f : \R_+ \rightarrow \R_+$ such that, if $g : (x,y) \in \R^2 \mapsto f\left(\sqrt{x^2 + y^2}\right)$ then:
\begin{enumerate}
\item $\int_{\R_+} f(r) \, r dr = \frac{1}{2\pi}$,
\item $\iint_{\R^2} g(x,y)\|(x,y)\| \, dx dy \leq \frac{\varepsilon}{48\pi\eth_0} = \frac{\frac{\varepsilon}{16}}{2\pi\frac{3\eth_0}{2}}$,
\item $\CDN_{\vartheta}^{p,q,d}(\sigma_{\eth,d}^g \xi) \leq (1 + \frac{\varepsilon}{16})\CDN_{\vartheta}^{p,q,d}(\xi)$ for all $\vartheta\not=\frac{p}{q}$ and all $\xi \in \mathcal{S}(\C^d)$.
\end{enumerate}
By \cite[Lemma 6.6]{Latremoliere17a}, we have that for all $\omega\in\mathcal{S}(\C^d)$, and for all $\eth \in \left[\frac{\eth_0}{2}, \frac{3\eth_0}{2}\right]$:
\begin{equation*}
\left\|\omega-\sigma_{\eth,d}^g\omega\right\|_{\HeisenbergMod{\frac{p}{q} + \eth}{p,q,d}} \leq \frac{\varepsilon}{16} \CDN_{\frac{p}{q} + \eth}^{p,q,d}(\omega)\text{,}
\end{equation*}
We apply Lemma (\ref{same-dim-approx-lemma}) to obtain some $N \in \N$ and $\delta_0 \in \left(0,\frac{|\eth_0|}{2}\right)$ such that if $\eth\in[\eth_0-\delta_0,\eth_0+\delta_0]$ then:
\begin{equation*}
\left\|f - \sum_{j=0}^N \mathrm{C}_\eth^j(f) \right\|_{L^1(\R_+, r dr)} \leq \frac{\varepsilon}{8} \text{.}
\end{equation*}

Let $h_\eth : (x,y) \in \R^2 \mapsto \sum_{j=0}^N \mathrm{C}_\eth^j(f) (\sqrt{x^2 + y^2})$ for all $\eth \not= 0$.

For each $\eth \in [\eth_0 - \delta_0,\eth_0 + \delta_0]$, we then have by \cite[Lemma 6.5]{Latremoliere17a}:
\begin{equation*}
\opnorm{ \sigma^g_{\eth,d} - \sigma^{h_\eth}_{\eth,d} }{}{\HeisenbergMod{\theta}{p,q,d}} \leq \|f - \sum_{j=0}^N \mathrm{C}_\eth^j(f)\|_{L^1(\R_+, r dr)} \leq \frac{\varepsilon}{8} \text{.}
\end{equation*}

Consequently, for all $\eth \in [\eth_0 - \delta_0,\eth_0 + \delta_0]$ and for all $\omega \in \HeisenbergMod{\eth+\frac{p}{q}}{p,q,d}$ such that $\CDN_{\eth+\frac{p}{q}}^{p,q,d}(\omega)\leq 1$, we have:
\begin{equation*}
\left\|\omega - \sigma^{h_\eth}_{\eth,d} \omega \right\|_{\HeisenbergMod{\eth+\frac{p}{q}}{p,q,d}} \leq \frac{3\varepsilon}{16} \text{,}
\end{equation*}
therefore $\CDN_{\frac{p}{q} + \eth}^{p,q,d}\left(\frac{1}{1 + \frac{\varepsilon}{16}} \sigma^{h_\eth}_{\eth,d} \omega \right) \leq 1$ and:
\begin{equation*}
\left\|\omega - \frac{1}{1 + \frac{\varepsilon}{16}} \sigma^{h_\eth}_{\eth,d} \omega \right\|_{\HeisenbergMod{\eth + \frac{p}{q}}{p,q,d}} \leq \frac{\frac{\varepsilon}{16}}{1 + \frac{\varepsilon}{16}}\|\omega\|_{\HeisenbergMod{}{p,q,d}} + \frac{3\varepsilon}{16} \leq \frac{\varepsilon}{4}\text{.}
\end{equation*}

By the proof of \cite[Lemma 6.9]{Latremoliere17a}, the range of $\sigma^{h_\eth}_{\eth,d}$ is of dimension $N+1$, spanned by:
\begin{equation*}
\left\{ \mathcal{H}_\eth^j = \eth^{\frac{1}{4}} \mathcal{H}_1^j\left(\sqrt{\eth}\cdot \right) : j \in \{0,\ldots,N\} \right\}
\end{equation*}
where:
\begin{equation*}
\mathcal{H}_1^j : t \in \R \mapsto \frac{\left(2\right)^{\frac{1}{4}}}{\sqrt{j! 2^j}} \exp\left(-\frac{ t^2 \sqrt{2 \pi} }{2}\right)\mathrm{H}_j\left(t \sqrt{2\pi}\right)
\end{equation*}
and $\operatorname{H}_j$ is the $j^{\mathrm{th}}$ Hermite polynomial.

Let $V = \C^N$. For any $\eth > 0$ we define $\eta_\eth : (c_j)_{j\in\{1,\ldots,N\}} \in V \mapsto \sum_{j=1}^N c_j \mathcal{H}_\eth^j$. The map $\eta_\eth$ is a linear injection from $V$ to $\mathcal{S}\otimes\C^d$. For each $\eth > 0$ and $c \in V$, we set $\|c\|_{\eth} = \CDN_{\frac{p}{q} + \eth}^{p,q,d}(\eta_\eth(c))$; of course $\|\cdot\|_\eth$ is a norm on $V$.

We now set $\|c\|_V = \sup_{\eth \in [\eth_0 - \delta_0,\eth_0 + \delta_0]} \|c\|_\eth$. By construction, it is sufficient to check that $\|\cdot\|_V$ is valued in $\R_+$ (i.e. is never infinite) to conclude that $\|\cdot\|_V$ is a norm on $V$.

Let $c = (c_j)_{j\in\{0,\ldots,N\}} \in V$. Note that for all $t\in \R$ and $\eth > 0$:
\begin{equation}\label{rescale-eq1}
\eta_\eth(c) (t) = \sum_{j=0}^N c_j \mathcal{H}_\eth^j(t) = \left(\eth\right)^{\frac{1}{4}} \sum_{j=0}^N \mathcal{H}_1^j(\sqrt{\eth}t) = \left(\eth\right)^{\frac{1}{4}} \eta_1(c) (\sqrt{\eth} t)\text{,}
\end{equation}
and of course, $\eta_1 \in \mathcal{S}(\C^d)$. Thus by Theorem (\ref{d-norm-cont-thm}), we conclude that
\begin{equation*}
  \eth \in [\eth_0 - \delta_0,\eth_0 + \delta_0]\mapsto \CDN_{\frac{p}{q} + \eth}^{p,q,d}(\eta_\eth(c)) = \sqrt[4]{\eth} \CDN_{\frac{p}{q}+\eth}^{p,q,d}(\sqrt{\eth} \eta_1(c))
\end{equation*}
is continuous.

Therefore, it reaches its maximum on the compact $[\eth_0 - \delta_0, \eth_0 + \delta_0]$, which is by definition the number $\|c\|_V$. Thus $\|\cdot\|_V$ is a norm on $V$.

We now make another observation. We have, for all $\eth \in [\eth_0 - \delta_0,\eth_0 + \delta_0]$, and for all $c,d \in V$:
\begin{equation*}
\begin{split}
\left| \|c\|_\eth - \|d\|_{\eth_0} \right| &\leq \left| \|c\|_\eth - \|c\|_{\eth_0}\right| + \left| \|c\|_{\eth_0} - \|d\|_{\eth_0} \right| \\
&\leq \left| \|c\|_\eth - \|c\|_{\eth_0}\right| + \left| \|c - d \right\|_{\eth_0} \\
&\leq \left| \|c\|_\eth - \|c\|_{\eth_0}\right| + \left| \|c - d \right\|_{V} \text{.} 
\end{split}
\end{equation*}

Thus the function:
\begin{equation*}
\mathsf{n} : (\eth,c)\in[\eth_0 - \delta_0,\eth_0 + \delta_0] \times V \mapsto\|c\|_\eth
\end{equation*}
is continuous. It is in particular continuous on the compact $[\eth_0 - \delta_0,\eth_0 + \delta_0]\times B$ where $B$ is the closed unit ball for $\|\cdot\|_V$. 

Therefore there exists $k > 0$ such that for all $c \in V$, $\eth \in [\eth_0 - \delta_0, \eth_0 + \delta_0]$, we have:
\begin{equation*}
k \|c\|_V \leq \|c\|_\eth \leq \|c\|_V \text{,}
\end{equation*}
where we emphasize that $k$ does not depend on $\eth$.

Let now $E = \{ c \in V : \|c\|_V \leq \frac{1}{k} \}$. Since $V$ is finite dimensional, $E$ is compact. Therefore, the function $\mathsf{n}$ is uniformly continuous on the compact $[\eth_0 - \delta_0,\eth_0 + \delta_0]\times E$. Let $\delta_1 \in (0,\delta_0)$ be chosen so that, for all $\eth \in [\eth_0 - \delta_1, \eth_0 + \delta_1]$, and for all $c,d \in E$ with $\|c - d\|_V \leq \delta_1$, we have:
\begin{equation*}
\left| \mathsf{n}(\eth,c) - \mathsf{n}(\eth_0,d) \right| \leq \frac{k \varepsilon}{8} \text{.}
\end{equation*}
In particular, for all $\eth \in [\eth_0 - \delta_1,\eth_0 + \delta_1]$ and all $c\in E$, we have:
\begin{equation*}
\left| \|c\|_\eth - \|c\|_{\eth_0} \right| \leq \frac{k \varepsilon}{8} \text{.}
\end{equation*}

Let:
\begin{equation*}
  \alg{E} = \frac{1}{1 + \frac{\varepsilon}{16}}\sigma_{\eth_0, d}^{h_{\eth_0}}\left(\left\{ \omega \in \HeisenbergMod{\eth+\frac{p}{q}}{p,q,d} : \CDN_{\frac{p}{q} + \eth_0}^{p,q,d}(\omega)\leq 1 \right\} \right)\text{.}
\end{equation*}
By definition, $\alg{E}$ is a bounded subset of $V$ which is finite dimensional. Thus $\alg{E}$ is totally bounded for $\|\cdot\|_V$. Let $\alg{F}$ be a finite $\frac{\varepsilon}{8}$-dense subset of $\alg{E}$ for $\|\cdot\|_V$. We assume $0 \not\in \alg{F}$ (we can simply pick a $\frac{\varepsilon}{16}$-dense subset of $\alg{E}$ and then remove $0$ from it if needed).

For all $c \in \alg{F}$, the function:
\begin{equation*}
\mathsf{l}_c : \eta \in [\eth_0 - \delta_1,\eth_0 + \delta_1] \mapsto \frac{\CDN_{\frac{p}{q} + \eth_0}^{p,q,d}(\eta_{\eth_0}(c)) - \CDN_{\frac{p}{q} + \eth}^{p,q,d}(\eta_{\eth_0}(c))}{\CDN_{\frac{p}{q} + \eth}^{p,q,d}(\eta_{\eth_0}(c))}
\end{equation*}
is continuous on a compact, and it is null at $\eth_0$; hence there exists $\delta_2 \in (0,\delta_1)$ such that for all $\eth \in [\eth_0 - \delta_2, \eth_0 + \delta_2]$ and $c\in\alg{F}$, we have:
\begin{equation*}
\mathsf{l}_c(\eth) \leq \frac{k \varepsilon}{4} \text{.}
\end{equation*}
We emphasize that in the definition of $\mathsf{l}_c$, for any $c\in \alg{F}$, only involves the element $\eta_{\eth_0}(c)$, and the only dependence on the variable is through the choice of D-norm.

Last, for all $d\in \alg{F}$, the function $\eth \in [\eth_0 - \delta_2, \eth_0 + \delta_2] \mapsto \|\eta_\eth(d) - \eta_{\eth_0}(d)\|_{\HeisenbergMod{\frac{p}{q} + \eth}{p,q,d}}$ is continuous by \cite[Proposition 3.5]{Latremoliere17a} and Expression (\ref{rescale-eq1}). Therefore, since $\alg{F}$ is finite, there exists $\delta_3 \in (0,\delta_2)$ such that for all $\eth \in [\eth_0 - \delta_3, \eth_0 + \delta_3]$ and for all $d\in \alg{F}$:
\begin{equation*}
\left\| \eta_{\eth}(d) - \eta_{\eth_0}(d) \right\|_{\HeisenbergMod{\frac{p}{q} + \eth}{p,q,d}} \leq \frac{\varepsilon}{4} \text{.}
\end{equation*}

Fix now $\eth \in [\eth_0 - \delta_3, \eth_0 + \delta_3]$. Let now $\omega \in \HeisenbergMod{\eth+\frac{p}{q}}{p,q,d}$ with $\CDN_{\frac{p}{q} + \eth}^{p,q,d}(\omega)\leq 1$. Let $c \in V$ so that $\frac{1}{1 + \frac{\varepsilon}{16}}\sigma_{\eth,d}^{h_\eth}(\omega) = \eta_{\eth}(c)$ and note that by construction, $\|\eta - \eta_\eth(c) \|_{\HeisenbergMod{\eth + \frac{p}{q}}{p,q,d}} \leq \frac{\varepsilon}{4}$ while $\CDN_{\frac{p}{q} + \eth}^{p,q,d}(\eta_\eth(c)) \leq 1$.

Since $\CDN_{\HeisenbergMod{\frac{p}{q} + \eth}{p,q,d}}^{p,q,d}(\eta_\eth(c)) \leq 1$, we conclude $\|c\|_V \leq \frac{1}{k}$. Thus, $\|\eta_{\eth_0}(c)\|_V \leq 1 + \frac{k \varepsilon}{8}$.

Thus, $\frac{1}{1+\frac{k}{8}\varepsilon}\eta_{\eth_0}(c) \in \alg{E}$, and thus there exists $d \in \alg{F}$ such that $\|\frac{1}{1+\frac{k}{8}\varepsilon}c - d\|_V \leq \frac{\varepsilon}{8}$. Thus:
\begin{equation*}
\|c - d\|_V \leq \frac{\frac{k}{8} \varepsilon}{1 + \frac{k}{8} \varepsilon}\|c\|_V + \frac{\varepsilon}{8} \leq \frac{\varepsilon}{4} \text{.}
\end{equation*}

We conclude by observing that:
\begin{multline*}
\left\| \frac{\CDN_{\frac{p}{q} + \eth_0}^{p,q,d}(\eta_{\eth_0}(d))}{\CDN_{\frac{p}{q}+\eth}^{p,q,d}(\eta_{\eth_0}(d))} \eta_{\eth_0}(d) - \omega \right\|_{\HeisenbergMod{\frac{p}{q} + \eth}{p,q,d}} \\
\begin{split}
&\leq \left\| \frac{\CDN_{\frac{p}{q} + \eth_0}^{p,q,d}(\eta_{\eth_0}(d))}{\CDN_{\frac{p}{q}+\eth}^{p,q,d}(\eta_{\eth_0}(d))} \eta_{\eth_0}(d) - \eta_{\eth}(c) \right\|_{\HeisenbergMod{\frac{p}{q} + \eth}{p,q,d}} + \frac{\varepsilon}{4} \\
&\leq \mathsf{l}_d(\eth) \|d\|_V + \|\eta_{\eth_0}(d) - \eta_{\eth}(c)\|_{\HeisenbergMod{\frac{p}{q} + \eth}{p,q,d}} + \frac{\varepsilon}{4} \\
&\leq \frac{\varepsilon}{4} + \|\eta_{\eth}(d) - \eta_{\eth_0}(d)\|_{\HeisenbergMod{\frac{p}{q} + \eth}{p,q,d}} + \|\eta_{\eth}(d) - \eta_{\eth}(c)\|_{\HeisenbergMod{\frac{p}{q} + \eth}{p,q,d}} + \frac{\varepsilon}{4} \\
&\leq \frac{\varepsilon}{4} + \frac{\varepsilon}{4} + \| c - d \|_V + \frac{\varepsilon}{4} \leq \varepsilon \text{.}
\end{split}
\end{multline*}

This concludes our lemma.
\end{proof}

We now turn to the proof of our main Theorem (\ref{main-qt-thm}).
\begin{notation}
Let $\mu$ be the probability Haar measure on the $2$-torus $\T^2$.

For any $f \in L^1(\T^2,\mu)$ and $\theta\in\R$, we denote by $\beta_\theta^f$ the operator on $\qt{\theta}$ defined for all $a\in\qt{\theta}$ by:
\begin{equation*}
\beta_\theta^f (a) = \int_{\T^2} f(z)\beta_\theta^z(a) \, d\mu(z)
\end{equation*}
which is continuous with $\opnorm{\beta_\theta^f}{}{\qt{\theta}}\leq \|f\|_{L^1(\T^2,\mu)}$.
\end{notation}

\begin{proof}[Proof of Theorem (\ref{main-qt-thm})]
Let $p \in \Z$ and $q \in \N\setminus\{0\}$. Let $d \in q\N\setminus\{0\}$. Let $X = \R\setminus\left\{\frac{p}{q}\right\}$. 

Let $\theta \in X$ and let $\varepsilon > 0$. By \cite{Latremoliere13c}, there exists $\delta_\varepsilon > 0$, a self-adjoint trace-class operator $T$ on $\ell^2(\Z^2)$ of norm $1$, a finite dimensional subspace $V$ of $\ell^1(\Z^2)$ and a nonnegative continuous function $\mathrm{Fe}:\T^2\rightarrow[0,\infty)$ such that if $\vartheta\in[\theta-\delta_{\varepsilon},\theta+\delta_{\varepsilon}]$, the range of $\beta^{\mathrm{Fe}}_\theta(\sa{\qt{\theta}}) = V$, such that the length of the bridge $(\alg{B}(\ell^2(\Z^2)),T,\pi_\theta,\pi_\vartheta)$ is no more than $\frac{\varepsilon}{16}$, where $\pi_\theta$ and $\pi_\vartheta$ are the GNS representations of $\qt{\theta}$ an $\qt{\vartheta}$, respectively, for the unique tracial states. Moreover, we record that $(a,\vartheta)\in V\times X \mapsto \Lip_{\vartheta}(a)$ is continuous.

To begin with, for all $\vartheta\in\R$, we note that if $a\in\ell^1(\Z^2)$, then $\beta_\vartheta^z(a)$ does not depend on $\vartheta \in \R$ for any $z\in \T^2$. Thus, the restriction of $\beta_\vartheta^{\mathsf{Fe}}$ to $\ell^1(\Z^2)$ is independent of $\vartheta$, valued in $V$, and will be denoted by $\beta^{\mathsf{Fe}}$.

By Lemma (\ref{anchors-lemma}), there exists a finite subset $F = \{\omega_j : j \in \{1,\ldots,N\} \}$ of $\modlip{1}{\HeisenbergMod{\theta}{p,q,d}}\setminus\{0\}$ for some $N\in\N$ and $\delta_0 > 0$ such that, for all $\vartheta \in [\theta - \delta_0, \theta + \delta_0]$, the set:
\begin{equation*}
\left\{ \frac{\CDN_\theta^{p,q,d}(\omega_j)}{\CDN_{\vartheta}^{p,q,d}(\omega_j)} \omega_j : j \in \{1,\ldots,N\} \right\}
\end{equation*}
is $\frac{\varepsilon}{16}$-dense in the closed unit ball of $\CDN_\vartheta^{p,q,d}$.

For each $\omega\in F$, the map $\vartheta \in X \mapsto \CDN_\vartheta^{p,q,d}(\omega)$ is continuous by Theorem (\ref{d-norm-cont-thm}). The function $\vartheta\mapsto \Lip_\vartheta(\inner{\omega}{\eta}{\HeisenbergMod{\theta}{p,q,d}})$ is also continuous for all $\omega,\eta \in F$ by \cite{Latremoliere05}. Last, for any $\omega \in F$, we note that the continuous function $\vartheta \in X \mapsto\CDN_\vartheta^{p,q,d}(\omega)$, reaches its minimum on the compact $[\theta-\delta_\varepsilon,\theta+\delta_\varepsilon]$, and thus in particular, since $\omega\not=0$ and $\CDN_\vartheta^{p,q,d}$ is a norm, this minimum is not zero.

Thus the functions:
\begin{equation*}
\mathsf{y}^\Re_{j,k} : \vartheta \in X \mapsto \frac{\Lip_\theta\left( \Re \beta^{\mathsf{Fe}}\inner{\omega_j}{\omega_k}{\HeisenbergMod{\theta}{p,q,d}} \right)}{\Lip_\vartheta\left( \Re \beta^{\mathsf{Fe}}\inner{\omega_j}{\omega_k}{\HeisenbergMod{\vartheta}{p,q,d}} \right)} - \frac{\CDN_{\theta}^{p,q,d}(\omega_j) \CDN_\theta^{p,q,d}(\omega_k)}{\CDN_\vartheta^{p,q,d}(\omega_j) \CDN_\vartheta^{p,q,d}(\omega_k)}
\end{equation*}
and
\begin{equation*}
\mathsf{y}^\Im_{j,k} : \vartheta \in X \mapsto \frac{\Lip_\theta\left( \Im \beta^{\mathsf{Fe}}\inner{\omega_j}{\omega_k}{\HeisenbergMod{\theta}{p,q,d}} \right)}{\Lip_\vartheta\left( \Im \beta^{\mathsf{Fe}}\inner{\omega_j}{\omega_k}{\HeisenbergMod{\vartheta}{p,q,d}} \right)} - \frac{\CDN_{\theta}^{p,q,d}(\omega_j) \CDN_\theta^{p,q,d}(\omega_k)}{\CDN_\vartheta^{p,q,d}(\omega_j) \CDN_\vartheta^{p,q,d}(\omega_k)}
\end{equation*}
are continuous as well for all $j,k \in \{1,\ldots,N\}$. Consequently, the function:
\begin{equation*}
\mathsf{y} = \max_{j,k \in \{1,\ldots,N\}} \left\{ \left| \mathsf{y}^\Re_{j,k} \right|, \left| \mathsf{y}^\Im_{j,k} \right| \right\}
\end{equation*}
is continuous as the maximum of finitely many continuous functions. We also note that $\mathsf{y}(\theta) = 0$. Thus there exists $\delta_2 > 0$ such that $|\mathsf{y}| < \frac{\varepsilon}{16}\text{ on }[\theta-\delta_2,\theta + \delta_2]$.

For each $j\in \{1,\ldots,N\}$, let:
\begin{equation*}
\eta_j = \frac{\CDN_\theta(\omega_j)}{\CDN_\vartheta(\omega_j)}\omega_j\text{.}
\end{equation*}

By construction, we have $\CDN_\vartheta(\eta_j) = \CDN_\theta(\omega_j)$.

Last, by \cite[Lemma 3.2]{Latremoliere17a}, there exists $\delta_3 > 0$ such that for all $\vartheta\in [\theta-\delta_3,\theta+\delta_3]$ we have, for all $j,k\in\{1,\ldots,N\}$:
\begin{equation*}
\left\|\inner{\eta_j}{\eta_k}{\HeisenbergMod{\theta}{p,q,d}} - \inner{\eta_j}{\eta_k}{\HeisenbergMod{\vartheta}{p,q,d}}\right\|_{\ell_1(\Z^2)} \leq \frac{\varepsilon}{16} \text{.}
\end{equation*} 

Let $\delta_4 = \min\{\delta_{\frac{\varepsilon}{16}},\delta_2,\delta_3\}$ and $\vartheta\in[\theta-\delta_4,\theta+\delta_4]$.

We now begin a string of inequalities for two given $j,k \in \{1,\ldots,N\}$. To begin with, by \cite[Proposition 3.8]{Latremoliere05}, for all $a\in\qt{\vartheta}$:
\begin{equation*}
\|a - \beta^{\mathsf{Fe}}a\|_{\qt{\vartheta}} \leq \|\Re a - \beta^{\mathsf{Fe}}\Re a\|_{\qt{\vartheta}} + \|\Im a - \beta^{\mathsf{Fe}}\Im a\|_{\qt{\vartheta}} \leq \frac{\varepsilon}{16}\left(\Lip_{\vartheta}(\Re a) + \Lip_{\vartheta}(\Im a)\right) \leq \frac{\varepsilon}{8}\Lip_{\vartheta}(a) \text{.}
\end{equation*}

Therefore, using the inner Leibniz inequality:
\begin{align}\label{main-eq1}
&\, \opnorm{\pi_\theta\left(\inner{\omega_j}{\omega_k}{\HeisenbergMod{\theta}{p,q,d}}\right) T - T \pi_\vartheta\left(\inner{\eta_j}{\eta_k}{\HeisenbergMod{\vartheta}{p,q,d}}\right)}{}{\ell^2(\Z^2)} \\ 
&\quad\leq \left\| \inner{\omega_j}{\omega_k}{\HeisenbergMod{\theta}{p,q,d}} - \beta^{\mathsf{Fe}}\inner{\omega_j}{\omega_k}{\HeisenbergMod{\theta}{p,q,d}} \right\|_{\qt{\theta}} \nonumber \\
&\quad\quad + \left\| \inner{\eta_j}{\eta_k}{\HeisenbergMod{\theta}{p,q,d}} - \beta^{\mathsf{Fe}}\inner{\eta_j}{\eta_k}{\HeisenbergMod{\theta}{p,q,d}} \right\|_{\qt{\vartheta}} \nonumber \\
&\quad\quad + \opnorm{\pi_\theta\left(\beta^{\mathsf{Fe}}\inner{\omega_j}{\omega_k}{\HeisenbergMod{\theta}{p,q,d}}\right) T - T \pi_\vartheta\left(\beta^{\mathsf{Fe}}\inner{\eta_j}{\eta_k}{\HeisenbergMod{\vartheta}{p,q,d}}\right)}{}{\ell^2(\Z^2)} \nonumber \\ 
&\quad\leq \frac{\varepsilon}{8}\left(\Lip_{\theta}\left(\inner{\omega_j}{\omega_k}{\HeisenbergMod{\theta}{p,q,d}}\right) + \Lip_{\vartheta}\left(\inner{\eta_j}{\eta_k}{\HeisenbergMod{\vartheta}{p,q,d}}\right)\right) \nonumber \\
&\quad\quad + \opnorm{\pi_\theta\left(\beta^{\mathsf{Fe}}\inner{\omega_j}{\omega_k}{\HeisenbergMod{\theta}{p,q,d}}\right) T - T \pi_\vartheta\left(\beta^{\mathsf{Fe}}\inner{\eta_j}{\eta_k}{\HeisenbergMod{\vartheta}{p,q,d}}\right)}{}{\ell^2(\Z^2)} \nonumber \\
&\leq \frac{\varepsilon}{8} (\CDN_\theta^{p,q,d}(\omega_j)\|\omega_k\|_{\HeisenbergMod{\theta}{p,q,d}} + \CDN_\theta^{p,q,d}(\omega_k)\|\omega_j\|_{\HeisenbergMod{\theta}{p,q,d}} \nonumber \\
&\quad\quad + \CDN_\vartheta^{p,q,d}(\eta_j)\|\eta_k\|_{\HeisenbergMod{\vartheta}{p,q,d}} + \CDN_\vartheta^{p,q,d}(\eta_k)\|\eta_j\|_{\HeisenbergMod{\vartheta}{p,q,d}}) \nonumber \\
&\quad\quad + \opnorm{\pi_\theta\left(\beta^{\mathsf{Fe}}\inner{\omega_j}{\omega_k}{\HeisenbergMod{\theta}{p,q,d}}\right) T - T \pi_\vartheta\left(\beta^{\mathsf{Fe}}\inner{\eta_j}{\eta_k}{\HeisenbergMod{\vartheta}{p,q,d}}\right)}{}{\ell^2(\Z^2)} \nonumber \\
&\leq \frac{\varepsilon}{2} + \opnorm{\pi_\theta\left(\beta^{\mathsf{Fe}}\inner{\omega_j}{\omega_k}{\HeisenbergMod{\theta}{p,q,d}}\right) T - T \pi_\vartheta\left(\beta^{\mathsf{Fe}}\inner{\eta_j}{\eta_k}{\HeisenbergMod{\vartheta}{p,q,d}}\right)}{}{\ell^2(\Z^2)} \nonumber \text{.}
\end{align}

Our next step is to replace $\beta^{\mathsf{Fe}}\inner{\eta_j}{\eta_k}{\HeisenbergMod{\vartheta}{p,q,d}}$ with $\beta^{\mathsf{Fe}}\inner{\eta_j}{\eta_k}{\HeisenbergMod{\theta}{p,q,d}}$, because our work in \cite{Latremoliere13c} follows a single element in $V$ from $\qt{\theta}$ to $\qt{\vartheta}$. 

Now, noting that $\beta^{\mathsf{Fe}}$ has norm $1$:
\begin{multline*}
\left\|\beta^{\mathsf{Fe}}\inner{\eta_j}{\eta_k}{\HeisenbergMod{\vartheta}{p,q,d}}\right\|_{\qt{\vartheta}}\\
\begin{split}
&\leq \left\|\beta^{\mathsf{Fe}}\inner{\eta_j}{\eta_k}{\HeisenbergMod{\theta}{p,q,d}}\right\|_{\qt{\vartheta}} + \left\|\beta^{\mathsf{Fe}}\inner{\eta_j}{\eta_k}{\HeisenbergMod{\vartheta}{p,q,d}} - \beta^{\mathsf{Fe}}\inner{\eta_j}{\eta_k}{\HeisenbergMod{\theta}{p,q,d}}\right\|_{\qt{\vartheta}} \\
&\leq \left\|\beta^{\mathsf{Fe}}\inner{\eta_j}{\eta_k}{\HeisenbergMod{\theta}{p,q,d}}\right\|_{\qt{\vartheta}} + \left\|\inner{\eta_j}{\eta_k}{\HeisenbergMod{\vartheta}{p,q,d}} - \inner{\eta_j}{\eta_k}{\HeisenbergMod{\theta}{p,q,d}}\right\|_{\qt{\vartheta}} \\
&\leq \left\|\beta^{\mathsf{Fe}}\inner{\eta_j}{\eta_k}{\HeisenbergMod{\theta}{p,q,d}}\right\|_{\qt{\vartheta}} + \left\|\inner{\eta_j}{\eta_k}{\HeisenbergMod{\vartheta}{p,q,d}} - \inner{\eta_j}{\eta_k}{\HeisenbergMod{\theta}{p,q,d}}\right\|_{\ell^1(\Z^2)}\\
&\leq \left\|\beta^{\mathsf{Fe}}\inner{\eta_j}{\eta_k}{\HeisenbergMod{\theta}{p,q,d}}\right\|_{\qt{\vartheta}} + \frac{\varepsilon}{16} \text{.}
\end{split}
\end{multline*}

Thus we conclude, as $\opnorm{T}{}{\ell^2(\Z^2)} = 1$:
\begin{multline}\label{main-eq2}
\opnorm{\pi_\theta\left(\beta^{\mathsf{Fe}}\inner{\omega_j}{\omega_k}{\HeisenbergMod{\theta}{p,q,d}}\right) T - T \pi_\vartheta\left(\beta^{\mathsf{Fe}}\inner{\eta_j}{\eta_k}{\HeisenbergMod{\vartheta}{p,q,d}}\right)}{}{\ell^2(\Z^2)}\\ \leq \opnorm{\pi_\theta\left(\beta^{\mathsf{Fe}}\inner{\omega_j}{\omega_k}{\HeisenbergMod{\theta}{p,q,d}}\right) T - T \pi_\vartheta\left(\beta^{\mathsf{Fe}}\inner{\eta_j}{\eta_k}{\HeisenbergMod{\theta}{p,q,d}}\right)}{}{\ell^2(\Z^2)} + \frac{\varepsilon}{16} \text{.}
\end{multline}
Inserting Inequality (\ref{main-eq2}) in Inequality (\ref{main-eq1}) we thus have:
\begin{multline}\label{main-eq3}
\opnorm{\pi_\theta\left(\inner{\omega_j}{\omega_k}{\HeisenbergMod{\theta}{p,q,d}}\right) T - T \pi_\vartheta\left(\inner{\eta_j}{\eta_k}{\HeisenbergMod{\vartheta}{p,q,d}}\right)}{}{\ell^2(\Z^2)} \\ 
\leq \frac{9\varepsilon}{16} + \opnorm{\pi_\theta\left(\beta^{\mathsf{Fe}}\inner{\omega_j}{\omega_k}{\HeisenbergMod{\theta}{p,q,d}}\right) T - T \pi_\vartheta\left(\beta^{\mathsf{Fe}}\inner{\eta_j}{\eta_k}{\HeisenbergMod{\theta}{p,q,d}}\right)}{}{\ell^2(\Z^2)} \text{.}
\end{multline}

Let now:
\begin{equation*}
s^\Re_{j,k} = s\left(\Re\inner{\omega_j}{\omega_k}{\HeisenbergMod{\theta}{p,q,d}}, \vartheta\right) = \frac{\Lip_\theta\left(\Re\inner{\omega_j}{\omega_k}{\HeisenbergMod{\theta}{p,q,d}}\right)}{\Lip_\vartheta\left(\Re\inner{\omega_j}{\omega_k}{\HeisenbergMod{\theta}{p,q,d}}\right)}
\end{equation*}
and:
\begin{equation*}
s^\Im_{j,k} = s\left(\Im\inner{\omega_j}{\omega_k}{\HeisenbergMod{\theta}{p,q,d}}, \vartheta\right) = \frac{\Lip_\theta\left(\Im\inner{\omega_j}{\omega_k}{\HeisenbergMod{\theta}{p,q,d}}\right)}{\Lip_\vartheta\left(\Im\inner{\omega_j}{\omega_k}{\HeisenbergMod{\theta}{p,q,d}}\right)}\text{.}
\end{equation*}
By \cite[Claim 5.15]{Latremoliere13c}, we record that $|1-s^\Re_{j,k}|<\frac{\varepsilon}{16}$ and $|1-s^\Im_{j,k}|<\frac{\varepsilon}{16}$.

We thus compute:
\begin{align}\label{main-eq4}
&\, \opnorm{\Re\left(\pi_\theta\left(\beta^{\mathsf{Fe}}\inner{\omega_j}{\omega_k}{\HeisenbergMod{\theta}{p,q,d}}\right) T - T \pi_\vartheta\left(\beta^{\mathsf{Fe}}\inner{\eta_j}{\eta_k}{\HeisenbergMod{\theta}{p,q,d}}\right)\right)}{}{\ell^2(\Z^2)} \nonumber \\
&\quad \leq \opnorm{\pi_\theta\left(\Re\beta^{\mathsf{Fe}}\inner{\omega_j}{\omega_k}{\HeisenbergMod{\theta}{p,q,d}}\right) T - T \pi_\vartheta\left(\frac{\CDN_\theta^{p,q,d}(\omega_j)\CDN_\theta^{p,q,d}(\omega_k)}{\CDN_\vartheta^{p,q,d}(\omega_j)\CDN_\vartheta^{p,q,d}(\omega_j)} \Re\beta^{\mathsf{Fe}}\inner{\omega_j}{\omega_k}{\HeisenbergMod{\theta}{p,q,d}}\right)}{}{\ell^2(\Z^2)} \nonumber
\\
&\quad \leq \opnorm{\pi_\theta\left(\Re\beta^{\mathsf{Fe}}\inner{\omega_j}{\omega_k}{\HeisenbergMod{\theta}{p,q,d}}\right) T - T \pi_\vartheta\left(s^\Re_{j,k} \Re\beta^{\mathsf{Fe}}\inner{\omega_j}{\omega_k}{\HeisenbergMod{\theta}{p,q,d}}\right)}{}{\ell^2(\Z^2)} \nonumber \\
&\quad\quad + \opnorm{\pi_\vartheta\left(\left(s^\Re_{j,k} - \frac{\CDN_\theta^{p,q,d}(\omega_j)\CDN_\theta^{p,q,d}(\omega_k)}{\CDN_\vartheta^{p,q,d}(\omega_j)\CDN_\vartheta^{p,q,d}(\omega_k)}\right) \Re\beta^{\mathsf{Fe}}\inner{\omega_j}{\omega_k}{\HeisenbergMod{\vartheta}{p,q,d}}\right)}{}{\ell^2(\Z^2)} \nonumber \\
&\quad \leq \opnorm{\pi_\theta\left(\Re\beta^{\mathsf{Fe}}\inner{\omega_j}{\omega_k}{\HeisenbergMod{\theta}{p,q,d}}\right) T - T \pi_\vartheta\left(s^\Re_{j,k} \Re\beta^{\mathsf{Fe}}\inner{\omega_j}{\omega_k}{\HeisenbergMod{\theta}{p,q,d}}\right)}{}{\ell^2(\Z^2)} \nonumber \\
  &\quad + \left\| \left(s^\Re_{j,k} - \frac{\CDN_\theta^{p,q,d}(\omega_j)\CDN_\theta^{p,q,d}(\omega_k)}{\CDN_\vartheta^{p,q,d}(\omega_j)\CDN_\vartheta^{p,q,d}(\omega_k)}\right) \inner{\omega_j}{\omega_k}{\HeisenbergMod{\vartheta}{p,q,d}} \right\|_{qt{\vartheta}} \text{.}
\end{align}

By assumption on $\vartheta$, we have:
\begin{equation}\label{main-eq5}
\left\| \left(s^\Re_{j,k} - \frac{\CDN_\theta^{p,q,d}(\omega_j)\CDN_\theta^{p,q,d}(\omega_k)}{\CDN_\vartheta^{p,q,d}(\omega_j)\CDN_\vartheta^{p,q,d}(\omega_k)}\right) \inner{\omega_j}{\omega_k}{\HeisenbergMod{\vartheta}{p,q,d}} \right\|_{\qt{\vartheta}} \leq \mathsf{y}(\vartheta) < \frac{\varepsilon}{16} \text{,}
\end{equation}
and thus, plugging Inequality (\ref{main-eq5}) in Inequality (\ref{main-eq4}), we obtain:
\begin{multline}\label{main-eq6}
\opnorm{\Re\left( \pi_\theta\left(\Re\beta^{\mathsf{Fe}}\inner{\omega_j}{\omega_k}{\HeisenbergMod{\theta}{p,q,d}}\right) T - T \pi_\vartheta\left(\Re\beta^{\mathsf{Fe}}\inner{\eta_j}{\eta_k}{\HeisenbergMod{\theta}{p,q,d}}\right) \right)}{}{\ell^2(\Z^2)} \\
\leq \opnorm{\pi_\theta\left(\Re\beta^{\mathsf{Fe}}\inner{\omega_j}{\omega_k}{\HeisenbergMod{\theta}{p,q,d}}\right) T - T \pi_\vartheta\left(s_{j,k} \Re\beta^{\mathsf{Fe}}\inner{\omega_j}{\omega_k}{\HeisenbergMod{\theta}{p,q,d}}\right)}{}{\ell^2(\Z^2)} + \frac{\varepsilon}{16} \text{.}
\end{multline}

The elements $\Re\beta^{\mathsf{Fe}}\inner{\omega_j}{\omega_k}{\HeisenbergMod{\vartheta}{p,q,d}}$, for all $\vartheta\in X$, lie in $V$. We now wish them to lie in $E = \ker\tau\cap V$ with $\tau : f\in\ell^1(\Z^2)\mapsto f(0)$ so that we may apply our work in \cite{Latremoliere13c}. Again to ease notations, let:
\begin{equation*}
\tau_{\vartheta}^{j,k} = \tau\left(\beta^{\mathsf{Fe}}\inner{\omega_j}{\omega_k}{\HeisenbergMod{\vartheta}{p,q,d}}\right)\text{.}
\end{equation*}
Of course, $\tau_{\vartheta}^{j,k} = \tau\left(\Re\beta^{\mathsf{Fe}}\inner{\omega_j}{\omega_k}{\HeisenbergMod{\vartheta}{p,q,d}}\right) = \tau\left(\Im\beta^{\mathsf{Fe}}\inner{\omega_j}{\omega_k}{\HeisenbergMod{\vartheta}{p,q,d}}\right)$.

We thus evaluate:
\begin{multline*}
\opnorm{\pi_\theta\left(\Re\beta^{\mathsf{Fe}}\inner{\omega_j}{\omega_k}{\HeisenbergMod{\theta}{p,q,d}}\right) T - T \pi_\vartheta\left(s_{j,k} \Re\beta^{\mathsf{Fe}}\inner{\omega_j}{\omega_k}{\HeisenbergMod{\theta}{p,q,d}}\right)}{}{\ell^2(\Z^2)} \\
\begin{split}
&\leq \opnorm{\pi_\theta\left(\Re\beta^{\mathsf{Fe}}\inner{\omega_j}{\omega_k}{\HeisenbergMod{\theta}{p,q,d}} - \tau_\theta^{j,k} \right) T - T \pi_\vartheta\left(s_{j,k} \Re\beta^{\mathsf{Fe}}\inner{\omega_j}{\omega_k}{\HeisenbergMod{\theta}{p,q,d}} - s_{j,k} \tau_\theta^{j,k} \right)}{}{\ell^2(\Z^2)} \\
&\quad + \left| \tau_\theta^{j,k} - s_{j,k}\tau_\theta^{j,k} \right|  \text{.}
\end{split}
\end{multline*}

Now $\left| \tau_\theta - s_{j,k}\tau_\theta\right| \leq |1 - s_{j,k}| \tau_{j,k} \leq |1 - s_{j,k}| < \frac{\varepsilon}{16}$ since $\tau_{j,k} \leq \|\inner{\omega_j}{\omega_k}{\HeisenbergMod{\theta}{p,q,d}}\|_{\qt{\vartheta}} \leq 1$.

We are now in the position to apply our work in \cite{Latremoliere13c}. To this end, let us make a simple observation. If $a\in E\setminus\{0\}$, and if $s(a,\vartheta) = \frac{\Lip_\theta(a)}{\Lip_\vartheta(a)} > 0$, then:
\begin{equation*}
\opnorm{\pi_\theta(a) T - T \pi_\vartheta(s(a)a)}{}{\ell^2(\Z^2)} \leq \Lip_{\vartheta}(a) \varepsilon \text{.}
\end{equation*}

Indeed, if $\Lip_{\vartheta}(a) = 0$ then $a \in \R\unit_{\A_\vartheta}$; as $a\in E$ we conclude that $a = 0$. Thus $s(a,\vartheta)$ is well-defined. 

We then note that $s(a) = s(ra)$ for any $r > 0$ by definition. Moreover, if $a\in E$ and $a\not=0$, then by \cite[Claim 5.15]{Latremoliere13c}:
\begin{multline*}
\opnorm{\pi_\vartheta(a) T - T \pi_\vartheta(s(a)a)}{}{\ell^2(\Z^2)} \\
\begin{split}
&= \|a\|_{\ell^1(\Z^2)}\opnorm{\pi_\theta\left(\frac{1}{\|a\|_{\ell^1(\Z^2)}} a\right) T - T \pi_\vartheta\left( \frac{s(\|a\|_{\ell^1(\Z^2)}^{-1}a)}{\|a\|_{\ell^1(\Z^2)}}a\right)}{}{\ell^2(\Z^2)} \\
&\leq \|a\|_{\ell^1(\Z^2)}\Lip_{\vartheta}\left( \frac{1}{\|a\|_{\ell^1(\Z^2)}}a\right) \varepsilon = \Lip_{\vartheta}(a) \varepsilon \text{.}
\end{split}
\end{multline*}

Returning to our main computation, we thus conclude:
\begin{multline}\label{main-eq8}
\opnorm{\pi_\theta\left(\Re\beta^{\mathsf{Fe}}\inner{\omega_j}{\omega_k}{\HeisenbergMod{\theta}{p,q,d}} - \tau_\theta^{j,k} \right) T - T \pi_\vartheta\left(s_{j,k} \Re\beta^{\mathsf{Fe}}\inner{\omega_j}{\omega_k}{\HeisenbergMod{\theta}{p,q,d}} - s_{j,k} \tau_\theta^{j,k} \right)}{}{\ell^2(\Z^2)} \\ \leq \frac{\varepsilon}{16} \text{.}
\end{multline}

We now insert Inequality (\ref{main-eq8}) into Inequality (\ref{main-eq6}) to conclude:
\begin{multline}\label{main-eq9}
\opnorm{\Re\left(\pi_\theta\left(\beta^{\mathsf{Fe}}\inner{\omega_j}{\omega_k}{\HeisenbergMod{\theta}{p,q,d}}\right) T - T \pi_\vartheta\left(\beta^{\mathsf{Fe}}\inner{\eta_j}{\eta_k}{\HeisenbergMod{\theta}{p,q,d}}\right)\right)}{}{\ell^2(\Z^2)} \leq \frac{3\varepsilon}{16} \text{.}
\end{multline}

We get the same inequality as Inequality (\ref{main-eq9}) for $\Im$ in place of $\Re$ by the same reasoning, so we get:
\begin{multline}\label{main-eq10}
\opnorm{\pi_\theta\left(\beta^{\mathsf{Fe}}\inner{\omega_j}{\omega_k}{\HeisenbergMod{\theta}{p,q,d}}\right) T - T \pi_\vartheta\left(\beta^{\mathsf{Fe}}\inner{\eta_j}{\eta_k}{\HeisenbergMod{\theta}{p,q,d}}\right)}{}{\ell^2(\Z^2)} \\
\begin{split}
&\leq \opnorm{\Re\left(\pi_\theta\left(\beta^{\mathsf{Fe}}\inner{\omega_j}{\omega_k}{\HeisenbergMod{\theta}{p,q,d}}\right) T - T \pi_\vartheta\left(\beta^{\mathsf{Fe}}\inner{\eta_j}{\eta_k}{\HeisenbergMod{\theta}{p,q,d}}\right)\right)}{}{\ell^2(\Z^2)} \\
&\quad + \opnorm{\Im\left(\pi_\theta\left(\beta^{\mathsf{Fe}}\inner{\omega_j}{\omega_k}{\HeisenbergMod{\theta}{p,q,d}}\right) T - T \pi_\vartheta\left(\beta^{\mathsf{Fe}}\inner{\eta_j}{\eta_k}{\HeisenbergMod{\theta}{p,q,d}}\right)\right)}{}{\ell^2(\Z^2)} \\
&\leq \frac{3\varepsilon}{8} \text{.}
\end{split}
\end{multline}

Thus inserting Inequality (\ref{main-eq10}) in Inequality (\ref{main-eq3}), we conclude:
\begin{multline}\label{qt-modular-reach-eq}
\opnorm{\pi_\theta\left(\inner{\omega_j}{\omega_k}{\HeisenbergMod{\theta}{p,q,d}}\right) T - T \pi_\vartheta\left(\inner{\eta_j}{\eta_k}{\HeisenbergMod{\vartheta}{p,q,d}}\right)}{}{\ell^2(\Z^2)} 
\leq \frac{9\varepsilon}{16} + \frac{3\varepsilon}{8} = \frac{15\varepsilon}{16} \text{.}
\end{multline}

By construction, the following is a modular bridge (note that $\opnorm{T}{}{\ell^2(\Z^2)} = 1$):
\begin{equation*}
\gamma_\vartheta = \left(\alg{B}(\ell^2(\Z)),T,\pi_\theta,\pi_\vartheta,(\omega_j)_{j\in\{1,\ldots,n\}},(\eta_j)_{j\in\{1,\ldots,n\}}\right)\text{.}
\end{equation*}

The length of the basic bridge $(\alg{B}(\ell^2(\Z^2),T,\pi_\theta,\pi_\vartheta))$ is no more than $\frac{\varepsilon}{16}$, so the basic reach and the height of $\gamma_\vartheta$ are bounded by $\frac{\varepsilon}{16}$. Now, Expression (\ref{qt-modular-reach-eq}) states that the modular reach of $\gamma$ is bounded above by $\frac{15\varepsilon}{16}$. Thus by Definition (\ref{bridge-length-def}), the reach of $\gamma_\vartheta$ is no more than $\frac{\varepsilon}{16} + \frac{15\varepsilon}{16} = \varepsilon$.

Our choice of anchors and co-anchors ensures that the imprint of $\gamma_\vartheta$ is no more than $\frac{\varepsilon}{16}$. Thus by Definition (\ref{bridge-length-def}), the length of $\gamma$ is no more than $\varepsilon = \max\left\{\varepsilon ,\frac{\varepsilon}{16}\right\}$. By Expression (\ref{propinquity-eq}), we conclude that:
\begin{multline*}
\modpropinquity{}\left(\left(\HeisenbergMod{\vartheta}{p,q,d}, \inner{\cdot}{\cdot}{\HeisenbergMod{\vartheta}{p,q,d}}, \CDN_\vartheta^{p,q,d}, \qt{\vartheta}, \Lip_{\vartheta} \right),\right. \\
\left. \left(\HeisenbergMod{\theta}{p,q,d}, \inner{\cdot}{\cdot}{\HeisenbergMod{\theta}{p,q,d}}, \CDN_\theta^{p,q,d}, \qt{\theta}, \Lip_{\theta}\right) \right) \leq \varepsilon \text{.}
\end{multline*}
This concludes our proof.
\end{proof}

\bibliographystyle{amsplain}

\providecommand{\bysame}{\leavevmode\hbox to3em{\hrulefill}\thinspace}
\providecommand{\MR}{\relax\ifhmode\unskip\space\fi MR }
\providecommand{\MRhref}[2]{%
  \href{http://www.ams.org/mathscinet-getitem?mr=#1}{#2}
}
\providecommand{\href}[2]{#2}

\vfill

\end{document}